\documentclass[a4paper,11pt,reqno]{amsart}

\usepackage[utf8]{inputenc}

\usepackage[a4paper,top=2cm,bottom=2cm,left=2.2cm,right=2.2cm]{geometry}

\usepackage{titlesec}
\usepackage{lipsum}
\usepackage{enumerate}
\usepackage{enumitem}
\usepackage{mathrsfs}
\usepackage{amsmath,amssymb,amsthm}
\usepackage{mathtools}
\usepackage{graphicx, float, subfigure}
\usepackage{url}
\usepackage{tikz}
\usetikzlibrary{matrix, backgrounds}

\usepackage{bbold}
\usepackage{mleftright}
\usepackage{color}
\usepackage{xcolor}
\usepackage{colortbl}
\usepackage{array}

\usepackage{ytableau}
\usepackage[all,cmtip]{xy}
\usepackage{multirow}
\usepackage{makecell}
\usepackage[colorinlistoftodos]{todonotes}
\usepackage[colorlinks=true, allcolors=red]{hyperref}
\usepackage{harpoon}
\usepackage{MnSymbol}
\usepackage{esvect}
\usepackage{diagbox}
\usepackage[title]{appendix}
\usepackage{tikz-cd}

\theoremstyle{plain}
\newtheorem{theorem}{\scshape Theorem}[section]

\newtheorem{lemma}[theorem]{\scshape Lemma}
\newtheorem{corollary}[theorem]{\scshape Corollary}

\newtheorem*{assumption*}{\scshape Assumption}

\newtheorem*{claim*}{Claim}

\theoremstyle{definition}

\newtheorem{definition}[theorem]{\scshape Definition}
\newtheorem{remark}[theorem]{\scshape Remark}
\newtheorem{example}[theorem]{\scshape Example}

\numberwithin{equation}{section}

\titleformat{\section}{\centering\bfseries}{\thesection}{1em}{\MakeUppercase}
\titleformat{\subsection}{\bfseries}{\thesubsection}{1em}{}

\begin{document}
\title{Richardson tableaux and Motzkin paths}

\author{Sen-Peng Eu, Tung-Shan Fu, Yi-Hao Kao, and Yi-Lin Lee}
\address{(S.-P. Eu) Department of Mathematics, National Taiwan Normal University, Taipei, Taiwan}
\email{speu@math.ntnu.edu.tw}

\address{(T.-S. Fu) Department of Applied Mathematics, National Pingtung University, Pingtung, Taiwan.}
\email{tsfu@mail.nptu.edu.tw}

\address{(Y.-H. Kao) Department of Mathematics, National Taiwan Normal University, Taipei, Taiwan}
\email{61440006s@ntnu.edu.tw}

\address{(Y.-L. Lee) Department of Mathematics, National Cheng Kung University, Tainan, Taiwan}
\email{yillee@gs.ncku.edu.tw}

\subjclass[2020]{05A15, 05A19, 05E10, 14M15}
\keywords{Motzkin paths; noncrossing involutions; Richardson varieties; Springer fibers; the Robinson--Schensted algorithm; Young tableaux.}

\begin{abstract}
    Richardson tableaux were introduced by Karp and Precup in their study of irreducible components of Springer fibers that are Richardson varieties. Guo gave an explicit bijection between Richardson tableaux and Motzkin paths through noncrossing involutions and the Robinson--Schensted (RS) algorithm. In this paper, we study Richardson tableaux of a fixed shape from the viewpoint of Motzkin paths. We introduce a shape algorithm, independent of the RS algorithm, that directly determines the shape of the corresponding Richardson tableau from a Motzkin path. Based on this algorithm, we construct a local bijection which keeps track of the major statistic and reproves combinatorially Karp and Precup's $q$-enumeration formula for Richardson tableaux of a given shape. We also prove in two ways a conjecture of Guo on the comajor generating function for Richardson tableaux with a prescribed number of odd columns.
\end{abstract}

\maketitle

\section{Introduction}\label{sec.intro}

Tableaux combinatorics is the study of tableaux and the rich network of bijections, symmetries, enumerative formulas, and dynamical operators that connect them to, for instance, representation theory, symmetric functions, and geometry of flag varieties. Broadly speaking, a \textit{tableau} is an arrangement of numbers, symbols, or objects that fill the boxes of a specific shape---most commonly a Young diagram---according to a set of rules. One important family of tableaux, called standard Young tableaux, is a fundamental object in algebraic and geometric combinatorics, which is introduced below.

A \textit{partition} $\lambda$ of $n$ is a sequence $(\lambda_1,\lambda_2,\dots,\lambda_{\ell})$ of integers such that $\lambda_1 \geq \lambda_2 \geq \cdots \geq \lambda_{\ell} > 0$ and $\lambda_1 + \lambda_2 + \cdots + \lambda_\ell= n$. We write $\lambda \vdash n$ or $|\lambda| = n$ (\textit{size} of $\lambda$). Sometimes we use the notation $k^i$ for $i$ consecutive parts equal to $k$. The \textit{conjugate} of $\lambda$, denoted $\mu = \lambda^t$, is the partition $(\mu_1,\mu_2,\dots,\mu_{\lambda_1})$ defined by $\mu_i=|\{j ~|~ \lambda_j \geq i\}|$. The \textit{Young diagram of shape $\lambda$} is a left-justified array of boxes with $\lambda_i$ boxes in the $i$th row (from top to bottom), presented in English notation. A \textit{standard Young tableau} (SYT) $T$ of shape $\lambda$ is a filling of the boxes of the Young diagram of shape $\lambda$ with $\{1,2,\dots,|\lambda|\}$ such that the numbers are strictly increasing along each row and each column. See Figure \ref{fig.SYTs} for examples. The number of SYTs of shape $\lambda$ equals the dimension of the corresponding irreducible representation of the symmetric group $\mathfrak{S}_{|\lambda|}$ indexed by $\lambda$. Geometrically, Spaltenstein \cite{Spa76} established that these tableaux index the irreducible components of the Springer fiber $\mathcal{B}_\lambda$. For a detailed combinatorial treatment of tableaux, see Stanley \cite{StanEC2}; for their applications to geometry and representation theory, see Fulton \cite{Fulton97}.

Richardson tableaux were introduced by Karp and Precup \cite{KP25} to index the irreducible components of Springer fibers that coincide with Richardson varieties (see \cite[Sections 1 and 2]{KP25}). Given an SYT $T$ of size $n$, let $r_j(T)$ be the row number in which the entry $j$ appears in $T$, for $1 \leq j \leq n$. We write $T_{<j}$ for the subtableau of $T$ which consists of the entries $1,2,\dots,j-1$. We call $T$ a \textit{Richardson tableau} if for each $1 \leq j \leq n$ with $r_j(T) \geq 2$, the largest entry of $T_{<j}$ in row $r_j(T)-1$ is greater than every entry of $T_{<j}$ in rows $i \geq r_j(T)$. Let $\mathcal{R}(n)$ (resp., $\mathcal{R}(\lambda)$) be the set of Richardson tableaux of size $n$ (resp., shape $\lambda$). For example, a Richardson tableau of size $20$ is shown in Figure \ref{fig.Richardsonexample}. For the tableau $T$ in Figure \ref{fig.Richardsonnonexample}, setting $j=9$ gives $r_9(T) = 2$ and $T_{<9}$ consists of the entries $1,2,\dots,8$. Note that the maximum entry of $T_{<9}$ in row $1$ is $6$, which is less than the entry $7$ in row $3$ (or $8$ in row $4$). Thus, $T$ is not a Richardson tableau.  
\begin{figure}[hbt!]
    \centering
    \subfigure[]{\label{fig.Richardsonexample}
    \begin{tikzpicture}
        \matrix (m) [matrix of nodes,
                     ampersand replacement=\&,
                     nodes={draw, minimum size=0.8cm, anchor=center},
                     column sep=-\pgflinewidth, row sep=-\pgflinewidth]{
            1 \& 4 \& 8 \& 10 \& 15 \& 17\\
            2 \& 5 \& 9 \& 16\\
            3 \& 11 \& 18\\
            6 \& 12 \& 19\\
            7 \& 13\\
            14\\
            20\\
        };
    \end{tikzpicture}
    }
    \hspace{20mm}
    \subfigure[]{\label{fig.Richardsonnonexample}
        \begin{tikzpicture}
            \matrix (m) [matrix of nodes,
                         ampersand replacement= \& ,
                         nodes={draw, minimum size=0.8cm, anchor=center},
                         column sep=-\pgflinewidth, row sep=-\pgflinewidth]{
                |[fill=yellow!30]|1 \& |[fill=yellow!30]|4 \& |[fill=red!30]|6 \& 10 \& 15 \& 17\\
                |[fill=yellow!30]|2 \& |[fill=yellow!30]|5 \& 9 \& 16\\
                |[fill=yellow!30]|3 \& |[fill=yellow!30]|7 \& 18\\
                |[fill=yellow!30]|8 \& 12 \& 19\\
                11 \& 13\\
                14\\
                20\\
            };
        \end{tikzpicture}
    }
    \label{fig.SYTs}
    \caption{Two standard Young tableaux of size $20$, where (a) is Richardson and (b) is not Richardson.}
\end{figure}

The main result in Karp and Precup's work provides various combinatorial and geometric characterizations of Richardson tableaux. In particular, they proved the following results. 
\begin{theorem}[{\cite[Theorem 4.4]{KP25}}]\label{thm.KP25-Moz}
    The number of Richardson tableaux of size $n$ is given by the $n$th Motzkin number.
\end{theorem}
\begin{theorem}[{\cite[Theorem 4.8]{KP25}}]\label{thm.KP25-q}
    Let $\lambda=(\lambda_1,\dots,\lambda_\ell)$ be a partition. Then the generating function of Richardson tableaux of shape $\lambda$ weighted by the major statistic is
    \begin{equation}\label{eq.Rq-analogue}
        \sum_{T \in \mathcal{R}(\lambda)}q^{\mathsf{maj}(T)} = q^{\sum_{2 \leq i \leq j \leq \ell} \lambda_i \lambda_j} \prod_{i=1}^{\ell-1} \begin{bmatrix}
            \lambda_i+\lambda_{i+2}+\lambda_{i+3}+\cdots + \lambda_{\ell} \\ \lambda_{i+1}+\lambda_{i+2}+\lambda_{i+3}+\cdots + \lambda_{\ell}
        \end{bmatrix}_q,
    \end{equation}
    where the $i=\ell - 1$ term of the product is $\begin{bmatrix}
        \lambda_{\ell-1}\\
        \lambda_\ell
    \end{bmatrix}_q$.
    In particular, when $q=1$, we obtain the number of Richardson tableaux of shape $\lambda$,
    \begin{equation}
        |\mathcal{R}(\lambda)| = \prod_{i=1}^{\ell-1} \binom{\lambda_i+\lambda_{i+2}+\lambda_{i+3}+\cdots + \lambda_{\ell} }{\lambda_{i+1}+\lambda_{i+2}+\lambda_{i+3}+\cdots + \lambda_{\ell}}.
    \end{equation}
\end{theorem}

In their work, the proof of Theorem \ref{thm.KP25-Moz} is algebraic, showing that these two numbers have the same generating function. On the other hand, the proof of Theorem \ref{thm.KP25-q} is more complicated and based on induction on the size of Richardson tableaux. In \cite[Problem 4.5]{KP25}, the authors posed the problem of finding an explicit bijection between Richardson tableaux of size $n$ and Motzkin paths of length $n$.

Later, Guo \cite{Guo25} answered the above problem by providing such a bijection. Specifically, he proved that, under the celebrated Robinson--Schensted algorithm, the insertion tableaux of \textit{noncrossing involutions} of $[n]:=\{1,2,\dots,n\}$ coincide with Richardson tableaux of size $n$. We summarize the result in the following theorem. 
\begin{theorem}[{\cite[Theorem 1.1]{Guo25}}]\label{thm.Guo}
    The set of insertion tableaux of noncrossing involutions of $[n]$ is in bijection with the set of Richardson tableaux of size $n$ via the Robinson--Schensted algorithm.
\end{theorem}
Moreover, these involutions can be naturally identified with Motzkin paths of length $n$. However, his bijection cannot give a combinatorial explanation of Theorem \ref{thm.KP25-q} (see \cite[Remark 1.8]{Guo25}). In this paper, our first result gives such a combinatorial explanation (Section \ref{sec.pfq-analogue}). In fact, we introduce a \textit{shape algorithm} (Section \ref{sec.shapealgo}), where the input is a Motzkin path $p$, and the output is the partition $\lambda$ that is exactly the shape of the corresponding Richardson tableau of $p$ under the Robinson--Schensted algorithm. Our shape algorithm consists of iterated simple modifications of $p$ and does not rely on the Robinson--Schensted algorithm. Furthermore, we construct a map (Section \ref{sec.localbij}) of inserting certain Motzkin paths into a Motzkin path and its inverse; this allows us to systematically construct the set of Motzkin paths which are in bijection with the set of Richardson tableaux of a given shape $\lambda$, and hence, prove Theorem \ref{thm.KP25-q} combinatorially.

We also obtain a counterpart result of the $q$-enumeration formula for $\mathcal{R}(\lambda)$ by the comajor statistic. 
\begin{corollary}\label{cor.qcomajor}
    Let $\lambda=(\lambda_1,\dots,\lambda_\ell)$ be a partition. Then
    \begin{equation}\label{eq.qcomajor}
        \sum_{T \in \mathcal{R}(\lambda)} q^{\mathsf{comaj}(T)} = q^{\frac{1}{2}\sum_{i=1}^{\ell} \lambda_i (\lambda_i-1)} \prod_{i=1}^{\ell-1}  \begin{bmatrix}
            \lambda_i + \lambda_{i+2} + \lambda_{i+3} + \cdots + \lambda_\ell \\
            \lambda_{i+1} + \lambda_{i+2} + \lambda_{i+3} + \cdots + \lambda_\ell
        \end{bmatrix}_q.
    \end{equation}
\end{corollary}

Guo considered a refinement of Richardson tableaux based on the number of odd columns, that is, columns containing an odd number of boxes. Let $\mathcal{R}(n,k)$ denote the subset of $\mathcal{R}(n)$, which consists of Richardson tableaux of size $n$ with $k$ odd columns, where $n \equiv k \pmod{2}$. Our second result proves Guo's conjecture (\cite[Conjecture 1.7]{Guo25}) on the generating function of $\mathcal{R}(n,k)$ by the comajor statistic, which is stated below.
\begin{theorem}\label{thm.Guoconj}
    For $0 \leq k \leq n$ with $n \equiv k \pmod{2}$,
    \begin{equation}\label{eq.Guoconj}
    \sum_{T \in \mathcal{R}(n,k)} q^{\mathsf{comaj}(T)} = q^{\binom{k}{2}} \begin{bmatrix}
        n \\ k
    \end{bmatrix}_q
    C_{\frac{n-k}{2}}(q),
\end{equation}
where $C_m(q) = \frac{1}{[m+1]_q} \begin{bmatrix}
    2m \\ m
\end{bmatrix}_q$ denotes the $m$th $q$-Catalan number.
\end{theorem}

We provide two proofs of Theorem \ref{thm.Guoconj}. The first one considers generalized Motzkin paths and identifies them with specific permutations. This identification transforms the comajor statistic into the standard major statistics on these permutations, allowing us to complete the proof by applying Stanley's shuffle theorem. In the second one, we construct a bijection that generates $\mathcal{M}(n,k)$ by inserting horizontal steps into Dyck paths while tracking the comajor statistic.

The rest of this paper is organized as follows. Section~\ref{sec.pre} provides the necessary background. In Section \ref{sec.shapeRichardson}, we investigate the set of Motzkin paths corresponding to the set of Richardson tableaux of shape $\lambda$ and give a combinatorial proof of Theorem \ref{thm.KP25-q} and Corollary \ref{cor.qcomajor}. In Section \ref{sec.Guoconj}, we present two proofs of Theorem \ref{thm.Guoconj}. Finally, Section \ref{sec.remarks} contains concluding remarks and further discussion.

\section{Preliminaries}\label{sec.pre}

We will use the following notations. For any positive integer $n$, define $[n]_q:= \frac{1-q^n}{1-q}$ and $[n]_q!:=[1]_q[2]_q \cdots [n]_q$ with $[0]_q:=0$ and $[0]_q!:=1$. The $q$-binomial coefficient $\begin{bmatrix}
    n \\ k
\end{bmatrix} _q$ is defined by $\frac{[n]_q!}{[k]_q! [n-k]_q!}$ for $n \geq k \geq 0$; otherwise, it is $0$. The $q$-multinomial coefficient $\begin{bmatrix}
    n \\ n_1,n_2,\dots,n_\ell
\end{bmatrix}_q$ is defined to be $\frac{[n]_q!}{[n_1]_q! [n_2]_q! \cdots [n_\ell]_q!}$ for nonnegative integers $n_i$'s whose sum equals $n$; otherwise, it is $0$. 

In this section, we introduce the combinatorial objects (Section \ref{sec.preinvmoz}) and their associated statistics (Section \ref{sec.prestat}) that will be discussed in this paper. We also recall the celebrated Robinson--Schensted algorithm (Section \ref{sec.preRS}) and the notion of concatenation and prime decompositions of these objects (Section \ref{sec.preprime}).

\subsection{Noncrossing involutions and Motzkin paths}\label{sec.preinvmoz}

Consider the symmetric group $\mathfrak{S}_n$. We say $\pi \in \mathfrak{S}_n$ is an \textit{involution} of length $n$ if $\pi^{-1} = \pi$. Expressing in cycle notation, an involution $\pi$ consists of $2$-cycles or singletons. It is natural to visualize an involution $\pi \in \mathfrak{S}_n$ by an \textit{arc diagram}: list $n$ dots labeled by $1,2,\dots,n$ from left to right, and draw an arc between $i$ and $j$ if $i,j$ belong to a $2$-cycle. In this case, we write $(i,j), i<j$ for an arc connecting $i$ and $j$. Two arcs $(a,b)$ and $(c,d)$ form a \textit{crossing} if $a<c<b<d$. An involution is \textit{noncrossing} if its arc diagram has no crossings. Let $\mathcal{I}(n)$ denote the set of noncrossing involutions of $[n]$ and $\mathcal{I}(n,k)$ denote the subset of $\mathcal{I}(n)$ containing those with $k$ singletons. Note that noncrossing involutions are also called \textit{noncrossing partial matchings} in \cite{Guo25}.

We take the following two involutions as an example:
\begin{align*}
    \pi_1 & = (1,7)(2,3)(4,6)(5)(8,20)(9)(10,14)(11,13)(12)(15,19)(16)(17,18),\\
    \pi_2 &= (1,11)(2,3)(4,20)(5,8)(6,14)(7)(9)(10,13)(12)(15,19)(16)(17,18).
\end{align*}
Figure \ref{fig.noncrossing} shows the arc diagram of $\pi_1$, which is a noncrossing involution with $4$ singletons, that is, $\pi_1 \in \mathcal{I}(20,4)$. Figure \ref{fig.involution} shows the arc diagram of 
$\pi_2$; in this case, some arcs cross.
\begin{figure}[hbt!]
    \centering
    \subfigure[]{\label{fig.noncrossing}
        \begin{tikzpicture}[scale=0.5]
            \foreach \i in {1,...,20}{
                \node[circle,fill=black,inner sep=0.5pt] (P\i) at (\i,0) {};
                \node[below] at (P\i) {\i};
            }
            \draw (P1) to[bend left=45] (P7);
            \draw (P2) to[bend left=45] (P3);
            \draw (P4) to[bend left=45] (P6);
            \draw (P8) to[bend left=45] (P20);
            \draw (P10) to[bend left=45] (P14);
            \draw (P11) to[bend left=45] (P13);
            \draw (P15) to[bend left=45] (P19);
            \draw (P17) to[bend left=45] (P18);
        \end{tikzpicture}
        }
        \subfigure[]{\label{fig.involution}
        \begin{tikzpicture}[scale=0.5]
            \foreach \i in {1,...,20}{
                \node[circle,fill=black,inner sep=0.5pt] (P\i) at (\i,0) {};
                \node[below] at (P\i) {\i};
            }
            \draw (P1) to[bend left=45] (P11);
            \draw (P2) to[bend left=45] (P3);
            \draw (P4) to[bend left=45] (P20);
            \draw (P5) to[bend left=45] (P8);
            \draw (P6) to[bend left=45] (P14);
            \draw (P10) to[bend left=45] (P13);
            \draw (P15) to[bend left=45] (P19);
            \draw (P17) to[bend left=45] (P18);
        \end{tikzpicture}
        }
    \subfigure[]{\label{fig.Motzkin}
    \begin{tikzpicture}[scale=0.6]
        \draw[thick, black]
            (0,0) -- (1,1)
                  -- (2,2) 
                  -- (3,1)
                  -- (4,2)
                  -- (5,2)
                  -- (6,1)
                  -- (7,0)
                  -- (8,1)
                  -- (9,1)
                  -- (10,2)
                  -- (11,3)
                  -- (12,3)
                  -- (13,2)
                  -- (14,1)
                  -- (15,2)
                  -- (16,2)
                  -- (17,3)
                  -- (18,2)
                  -- (19,1)
                  -- (20,0);
        \foreach \x/\y in {0/0, 1/1, 2/2, 3/1, 4/2, 5/2, 6/1, 7/0, 8/1, 9/1, 10/2, 11/3, 12/3, 13/2, 14/1, 15/2, 16/2, 17/3, 18/2, 19/1, 20/0}{
            \filldraw[black] (\x,\y) circle (2pt);
        }
    \end{tikzpicture}   
    }
    \caption{(a) The arc diagram of the noncrossing involution $\pi_1$ of length $20$. (b) The arc diagram of the involution $\pi_2$ of length $20$. (c) The Motzkin path of length $20$ corresponding to the arc diagram in Figure \ref{fig.noncrossing}.}
\end{figure}
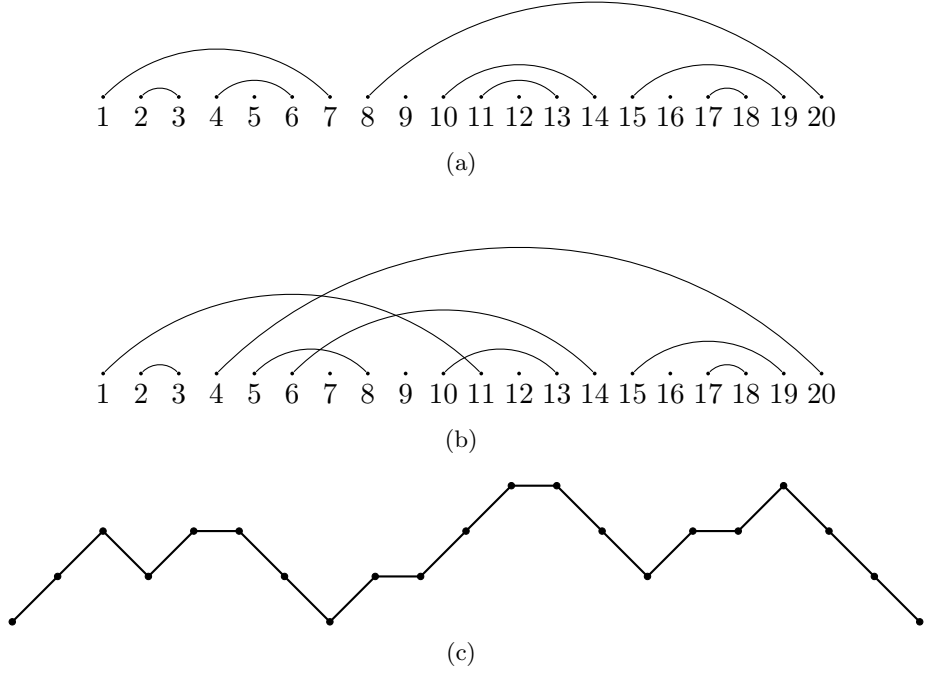

A \textit{Motzkin path} of length $n$ is a lattice path starting from $(0,0)$ and ending at $(n,0)$, which stays weakly above the $x$-axis, with steps $(1,1)$, $(1,0)$, and $(1,-1)$. For brevity, we denote these three steps by $U,H$, and $D$, respectively. Let $\mathcal{M}(n)$ be the set of Motzkin paths of length $n$ and $\mathcal{M}(n,k)$ be the subset of $\mathcal{M}(n)$ containing paths with exactly $k$ $H$ steps. 

There is a natural bijection between $\mathcal{I}(n,k)$ and $\mathcal{M}(n,k)$. To see this, given a noncrossing involution $\pi \in \mathcal{I}(n)$, the corresponding Motzkin path is constructed as follows. For $1 \leq i \leq n$, the $i$th step of the path is a $U$ (resp., $D$) step if $i$ is a left (resp., right) endpoint of an arc, and is an $H$ step if $i$ is a singleton. See Figure \ref{fig.Motzkin} for an illustration.

\subsection{The Robinson--Schensted algorithm}\label{sec.preRS}

The classical Robinson--Schensted (RS) algorithm sends a permutation $\pi$ of length $n$ to a pair of standard Young tableaux $(P, Q)$ of size $n$ with the same shape. Here, $P$ (resp., $Q$) is called the \textit{insertion} (resp., \textit{recording}) tableau of $\pi$. We will only use this classical version in the paper; see, for instance, \cite[Chapter 3]{Sagan01} for its generalizations.   

Given a permutation $\pi=\pi_1\pi_2\cdots \pi_n$, to describe the RS algorithm, we construct a sequence of tableaux pairs $\{ (P_i,Q_i) | i=0,1,\dots,n\}$, where $(P_0,Q_0)=(\emptyset, \emptyset)$ is the pair of empty tableaux and $(P_n,Q_n)=(P,Q)$ is the pair of resulting tableaux. $P_i$ is obtained by inserting $\pi_i$ into $P_{i-1}$ and the entry $i$ is placed in $Q_i$ so that $P_i$ and $Q_i$ have the same shape.

We describe the operation of inserting $\pi_i$ into $P_{i-1}$ below. Set $R$ to be the first row of $P_{i-1}$ and $x=\pi_i$. 
\begin{itemize}
    \item While $x$ is smaller than some entries of $R$, do
    \begin{itemize}
        \item Find the smallest entry $\ell$ greater than $x$, and replace $\ell$ by $x$ in $R$. We denote this operation by $R \leftarrow x$.
        \item Set $R$ to be the next row of $P_{i-1}$ and $x=\ell$. 
    \end{itemize}
    \item If $x$ is greater than all entries of $R$ or $R$ is empty, append $x$ to the end of this row and stop. 
\end{itemize}

Sch{\"u}tzenberger \cite{Schutz63} proved that $\pi$ is an involution if and only if $P=Q$ under the RS algorithm. In other words, the RS algorithm gives a bijection between the set of involutions of $[n]$ and the set of SYTs of size $n$. Sch{\"u}tzenberger  \cite{Schutz77} also showed that the number of fixed points of an involution is the same as the number of odd columns in the corresponding insertion tableau. 

\begin{example}\label{ex.RS}
    Let $\pi=413265$ be a permutation of length $6$. We list the tableau pairs $(P_i,Q_i)$ of applying the RS algorithm to $\pi$ below. The insertion tableau of $\pi$ is given by $P_6$ while the recording tableau of $\pi$ is given by $Q_6$. 
    \begin{align*}
        (\emptyset,\emptyset) & \rightarrow \left(~ \ytableausetup{centertableaux} \begin{ytableau} 4 \end{ytableau}, \ytableausetup{centertableaux} \begin{ytableau} 1 \end{ytableau}~\right) 
        \rightarrow \left(~
        \ytableausetup{centertableaux} \begin{ytableau} 1 \\ 4 \end{ytableau},\ytableausetup{centertableaux} \begin{ytableau} 1 \\ 2 \end{ytableau} ~\right)
        \rightarrow \left(~
        \ytableausetup{centertableaux} \begin{ytableau} 1 & 3 \\ 4 \end{ytableau}, \ytableausetup{centertableaux} \begin{ytableau} 1 & 3 \\ 2 \end{ytableau} ~\right)
        \rightarrow \left(~
        \ytableausetup{centertableaux} \begin{ytableau} 1 & 2 \\ 3 \\ 4 \end{ytableau}, \ytableausetup{centertableaux} \begin{ytableau} 1 & 3 \\ 2 \\ 4 \end{ytableau} ~\right) \\
        & \rightarrow \left(~
        \ytableausetup{centertableaux} \begin{ytableau} 1 & 2 & 6 \\ 3 \\ 4 \end{ytableau}, \ytableausetup{centertableaux} \begin{ytableau} 1 & 3 & 5 \\ 2 \\ 4 \end{ytableau} ~\right)
        \rightarrow \left(~
        \ytableausetup{centertableaux} \begin{ytableau} 1 & 2 & 5 \\ 3 & 6 \\ 4 \end{ytableau}, \ytableausetup{centertableaux} \begin{ytableau} 1 & 3 & 5 \\ 2 & 6 \\ 4 \end{ytableau} ~\right).
    \end{align*}
    We encourage the reader to check that applying the RS algorithm to $\pi_1$ in Figure \ref{fig.noncrossing} gives the standard Young tableau shown in Figure \ref{fig.Richardsonexample}. Theorem \ref{thm.Guo} asserts that this is a Richardson tableau because $\pi_1$ is a noncrossing involution. Similarly, applying the RS algorithm to $\pi_2$ in Figure \ref{fig.involution} gives the standard Young tableau shown in Figure \ref{fig.Richardsonnonexample}, but it is not a Richardson tableau. 
\end{example}

We would like to point out that when $\pi=n\cdots321$ is the reverse of the identity permutation (it is an involution), the RS algorithm sends $\pi$ to the SYT with only one column containing $1,2,\dots,n$ from top to bottom.

\subsection{Statistics on permutations, Motzkin paths, and standard Young tableaux}\label{sec.prestat}

Let $\pi=\pi_1\pi_2 \cdots \pi_n$ be a permutation of length $n$. We say $i \in [n-1]$ is a \textit{descent} (resp., \textit{ascent}) of $\pi$ if $\pi_i > \pi_{i+1}$ (resp., $\pi_i < \pi_{i+1}$). We denote by $\mathsf{Des}(\pi)$ (resp., $\mathsf{Asc}(\pi)$) the descent set (resp., ascent set) of $\pi$, which contains all descents (resp., ascents) of $\pi$. The \textit{major} statistic is defined as the sum of elements in the descent set, $\mathsf{maj}(\pi) = \sum_{i \in \mathsf{Des}(\pi)} i$, while the \textit{comajor} statistic is defined to be the sum of elements in the ascent set, $\mathsf{comaj}(\pi) = \sum_{i \in \mathsf{Asc}(\pi)} i$.

For a standard Young tableau $T$ of size $n$, we say $i \in [n-1]$ is a \textit{descent} of $T$ if $i+1$ appears in a row strictly below $i$ in $T$. The number $i \in [n-1]$ is an \textit{ascent} of $T$ if $i$ is not a descent of $T$. We write $\mathsf{Des}(T)$ (resp., $\mathsf{Asc}(T)$) for the descent set (resp., ascent set) of $T$, which contains all descents (resp., ascents) of $T$. Similarly, the \textit{major} and \textit{comajor} statistics of $T$ are given by $\mathsf{maj}(T) = \sum_{i \in \mathsf{Des}(T)} i$ and $\mathsf{comaj}(T) = \sum_{i \in \mathsf{Asc}(T)} i$, respectively.

We turn our attention to noncrossing involutions, or equivalently, Motzkin paths (Section \ref{sec.preinvmoz}). Let $p \in \mathcal{M}(n)$ be a Motzkin path of length $n$. We may express $p$ as a sequence of letters $U, D$, and $H$. We say $i \in [n-1]$ is a \textit{weak valley} of $p$ if the positions $i$ and $i+1$ in the sequence are of the form $HH$, $DU$, $DH$, or $HU$. Dually, we say $i \in [n-1]$ is a \textit{weak peak} of $p$ if the positions $i$ and $i+1$ in the sequence are of the form $UD$, $UU$, $DD$, $UH$, or $HD$. 
For convenience, we write $\mathsf{Val}(p)$ and $\mathsf{Pk}(p)$ for the set of weak valleys and weak peaks of $p$, respectively.

Barnabei, Bonetti, Castronuovo, and Silimbani \cite[Corollary 3.3]{BBCS19} gave a characterization of the ascent set of a noncrossing involution\footnote{Noncrossing involutions are equivalent to $3412$-avoiding permutations in the work of \cite{BBCS19}.} in terms of specific patterns of Motzkin paths, which is stated below.
\begin{lemma}[{\cite{BBCS19}}]\label{lem.weakvally}
    The index $i$ is an ascent of a noncrossing involution if and only if $i$ is a weak valley of the corresponding Motzkin path. 
\end{lemma}
We can easily deduce from Lemma \ref{lem.weakvally} that the index $i$ is a descent of a noncrossing involution if and only if $i$ is a weak peak of the corresponding Motzkin path. From this point of view, we may define the major and comajor statistics on the set of Motzkin paths. Given $p \in \mathcal{M}(n)$, the major of $p$ is defined as $\mathsf{maj}(p) = \sum_{i \in \mathsf{Pk}(p)}i$ and the comajor of $p$ is defined to be $\mathsf{comaj}(p) = \sum_{i \in \mathsf{Val}(p)}i$.

It is known (see, for example, \cite[Lemma 7.23.1]{StanEC2}) that the index $i$ is a descent (resp., ascent) of $\pi$ if and only if $i$ is a descent (resp., ascent) of $T$, where $T$ is the insertion tableau of $\pi$ under the RS algorithm. So, $\mathsf{maj}(\pi) = \mathsf{maj}(T)$, and $\mathsf{comaj}(\pi) = \mathsf{comaj}(T)$. Now, if we restrict to noncrossing involutions and Richardson tableaux, then by Theorem \ref{thm.Guo}, the index $i$ is a descent (resp., ascent) of a noncrossing involution if and only if $i$ is a descent (resp., ascent) of the corresponding Richardson tableau under the RS algorithm.

For example, we consider the noncrossing involution $\pi_1$ given in Section \ref{sec.preinvmoz} (its arc diagram is shown in Figure \ref{fig.noncrossing}), the corresponding Motzkin path $p_1$ (Figure \ref{fig.Motzkin}), and the insertion tableau $T_1$ (Figure \ref{fig.Richardsonexample}) of $\pi_1$ under the RS algorithm. Following the above discussion, one can check that the ascent set of $\pi_1$, the set of weak valleys of $p_1$, and the ascent set of $T_1$ are identical:
\begin{equation*}
    \mathsf{Asc}(\pi_1) = \mathsf{Val}(p_1) = \mathsf{Asc}(T_1) = \{3,7,9,14,16\}.
\end{equation*}
Thus, the comajor statistic $\mathsf{comaj}(\pi_1)= \mathsf{comaj}(p_1)=\mathsf{comaj}(T_1) = 49$.

\subsection{Concatenation and prime decompositions}\label{sec.preprime}

The notion of concatenation and prime decomposition of Richardson tableaux (\cite[Section 3.3]{KP25}) has a natural geometric meaning on the irreducible components of Springer fibers as shown by Fresse \cite{Fresse11}. Since the set of Richardson tableaux is in bijection with the set of noncrossing involutions (Theorem \ref{thm.Guo}), we can extend the definition of concatenation and prime decomposition to noncrossing involutions, or equivalently, Motzkin paths (\cite[Section 3]{Guo25}). We now introduce them below. 

Let $\pi \in \mathfrak{S}_m$ and $\sigma \in \mathfrak{S}_n$ be two permutations. The \textit{direct sum} $\pi \oplus \sigma$ of $\pi$ and $\sigma$ is the permutation in $\mathfrak{S}_{m+n}$ given by
\begin{equation*}
    \pi_1,\pi_2,\dots, \pi_m, \sigma_1+m,\sigma_2+m,\dots,\sigma_n+m. 
\end{equation*}
In other words, we concatenate $\pi$ and $\sigma$ and then increase each entry of $\sigma$ by $m$. For instance, if $\pi=321$ and $\sigma=632541$, then $\pi \oplus \sigma = 321965874$. A permutation is called \textit{prime} if it cannot be written as a direct sum of non-empty shorter permutations. As explained in \cite[Remark 3.1]{Guo25}, a noncrossing involution $w \in \mathcal{I}(n)$ is prime if and only if there is an arc $(1,n)$ in the arc diagram of $w$. Also, each noncrossing involution is a direct sum of prime noncrossing involutions.

Let $p$ and $q$ be two Motzkin paths of lengths $m$ and $n$, respectively. Similarly, the direct sum $p \oplus q $ of $p$ and $q$ is the Motzkin path of length $m+n$ defined by concatenating the ending point of $p$ with the starting point of $q$. A \textit{block} of a Motzkin path is one of the following two types:
\begin{enumerate}
    \item a single horizontal step on the $x$-axis;
    \item a sequence of steps that starts with an up step from the $x$-axis and terminates with the first down step that returns to the $x$-axis afterward.
\end{enumerate}
Sometimes the $x$-axis is called the \textit{base line} of the block. A Motzkin path is called \textit{prime} if it consists of a single block.

Let $T$ and $U$ be two standard Young tableaux of size $m$ and $n$, respectively. The \textit{concatenation} $T \circ U$ of $T$ and $U$ is the standard Young tableau of size $m+n$ given by increasing each entry of $U$ by $m$, and then concatenating the corresponding rows of $T$ and $U$. See Figure \ref{fig.concatenation} for an example. A standard Young tableau is called \textit{prime} if it cannot be written as a concatenation of non-empty smaller standard Young tableaux. Guo \cite[Lemma 3.3]{Guo25} observed that if $T$ and $U$ are the insertion tableaux of involutions $\pi$ and $\sigma$, respectively, then the insertion tableau of $\pi \oplus \sigma$ is equal to $T \circ U$. 
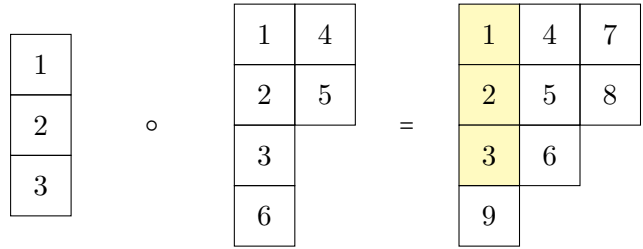
\begin{figure}[H]
    \centering
    \begin{minipage}{0.15\textwidth}
        \centering
        \begin{tikzpicture}[scale=0.5]
            \matrix (m) [matrix of nodes,
                         ampersand replacement=\&,
                         nodes={draw, minimum size=0.8cm, anchor=center},
                         column sep=-\pgflinewidth, row sep=-\pgflinewidth]{
                1 \\
                2 \\
                3 \\
            };
        \end{tikzpicture}
    \end{minipage}
    $\circ$
    \hspace{3mm}
    \begin{minipage}{0.15\textwidth}
        \centering
        \begin{tikzpicture}[scale=0.5]
            \matrix (m) [matrix of nodes,
                         ampersand replacement=\&,
                         nodes={draw, minimum size=0.8cm, anchor=center},
                         column sep=-\pgflinewidth, row sep=-\pgflinewidth]{
                1 \& 4 \\
                2 \& 5 \\
                3 \\
                6 \\
            };
        \end{tikzpicture}
    \end{minipage}
    $=$
    \begin{minipage}{0.2\textwidth}
        \centering
        \begin{tikzpicture}[scale=0.5]
            \matrix (m) [matrix of nodes,
                         ampersand replacement=\&,
                         nodes={draw, minimum size=0.8cm, anchor=center},
                         column sep=-\pgflinewidth, row sep=-\pgflinewidth]{
                |[fill=yellow!30]|1 \& 4 \& 7 \\
                |[fill=yellow!30]|2 \& 5 \& 8 \\
                |[fill=yellow!30]|3 \& 6\\
                9 \\
            };
        \end{tikzpicture}
    \end{minipage}
    \caption{The concatenation of two standard Young tableaux.}
    \label{fig.concatenation}
\end{figure}

\section{The shape of Richardson tableaux}\label{sec.shapeRichardson}

We study the shapes of Richardson tableaux via Motzkin paths. In Section~\ref{sec.shapealgo}, we introduce a shape algorithm for Motzkin paths that gives the shape of the corresponding Richardson tableau. In Section~\ref{sec.localbij}, we construct a bijection to generate all Motzkin paths whose associated Richardson tableaux have a given shape. We prove Theorem \ref{thm.KP25-q} and Corollary \ref{cor.qcomajor} in Section \ref{sec.pfq-analogue}.

\subsection{The shape of Motzkin paths}\label{sec.shapealgo}

Note that the permutation $\sigma=(n,n-1,\dots,1)\in\mathfrak{S}_n$ is the unique noncrossing involution that is associated with a Richardson tableau consisting of a single column. We consider the corresponding Motzkin path.

\begin{definition} {\rm
    For $n\geq 1$, let $m=\lfloor n/2\rfloor$. A path $p\in\mathcal{M}(n)$ is called a \textit{unit Motzkin path} if 
    \begin{equation*}
        p=\begin{cases}
            U^mD^m & \mbox{if $n$ even;}\\
            U^mHD^m & \mbox{if $n$ odd.}
            \end{cases}
    \end{equation*}
Notice that for $n=1$, $p$ is a single horizontal step on the $x$-axis. A unit Motzkin path of length $n$ is denoted by $u^{(n)}$. For convenience, a maximal sequence of steps of the form $U^kD^k$ (resp., $U^kHD^k$) for some $k\geq 0$ in a Motzkin path is called a \textit{unital segment} of even (resp., odd) length of the path.
}
\end{definition}

We first define some terminology on Motzkin paths. Let $p \in \mathcal{M}(n)$ be a Motzkin path of length $n$. We write $y(p)$ for the $y$-coordinate of the highest lattice point of $p$. The \textit{height} of $p$ is defined to be
\begin{equation}
    \mathsf{ht}(p) = \begin{cases}
        y(p)+\frac{1}{2}, & \text{if $p$ contains an $H$ step on the line $y=y(p)$,}\\
        y(p), & \text{otherwise.}
    \end{cases}
\end{equation}

Clearly, under the RS algorithm, a unit Motzkin path $p$ corresponds to a Richardson tableau of a single column of size $2 \mathsf{ht}(p)$. This motivates the following \textit{shape algorithm}. Using this algorithm, we shall establish a surjective map $\theta:\mathcal{M}(n)\rightarrow\{\lambda\mid\lambda\vdash n\}$ such that a Motzkin path $p\in\mathcal{M}(n)$ is carried to a partition $\lambda$ of $n$, which is exactly the shape of the Richardson tableau corresponding to $p$. Our algorithm consists of simple modifications of subpaths and does not rely on the RS algorithm. The shape algorithm is presented below.

\noindent \textbf{The shape algorithm.} 

Given a Motzkin path $p\in\mathcal{M}(n)$, set $\overline{p} = p$. The corresponding partition $\theta(p)$ of $n$ is constructed recursively as follows.
\begin{itemize}
    \item[(1)] Let $\overline{p}=p_1 \oplus p_2 \oplus \cdots \oplus p_s$ be the block decomposition of $\overline{p}$ for some $s\geq 1$. If all of $p_1, p_2, \dots, p_s$ are unit Motzkin paths, then go to (3); otherwise, go to (2).
    \item[(2)] Let $i$ be the least integer such that $p_i$ is not a unit Motzkin path. Let $m$ be the greatest integer such that the block $p_i$ can be decomposed as
    \begin{equation}\label{eq.primedecomp}
        p_i = U^m \oplus q_1 \oplus q_2 \oplus \cdots \oplus q_k \oplus D^m, 
    \end{equation}
    where each $q_j$ is a block with the base line $y=m$, for $1\leq j\leq k$. Let $t$ be the least integer such that $\mathsf{ht}(q_t)$ is one of the greatest numbers among $\mathsf{ht}(q_1), \mathsf{ht}(q_2), \dots, \mathsf{ht}(q_k)$. Then we set 
    \begin{equation}\label{eq.interchange}
        p_i^{\prime} = q_1 \oplus \cdots \oplus q_{t-1} \oplus U^m \oplus q_t \oplus D^m \oplus q_{t+1} \oplus \cdots \oplus q_k. 
    \end{equation}
    Update $\overline{p} = p_1 \oplus \cdots \oplus p_{i-1} \oplus p_i^{\prime} \oplus p_{i+1} \oplus \cdots \oplus p_s$ and go back to $(1)$. 
    \item[(3)] The resulting Motzkin path is denoted by $p_{\mathsf{min}} = \bar{p}$. For $1\leq i\leq s$, let $\mu_i$ be the length of the unit Motzkin path $p_i$. Then we set $\mu$ to be the partition of $n$ obtained by arranging $\mu_1,\mu_2,\dots,\mu_s$ in decreasing order, and set $\theta(p) = \mu^{t}$ to be the conjugate of $\mu$.
\end{itemize}

We remark that in \eqref{eq.primedecomp}, there are possibly multiple blocks $q_j$ reaching the maximum $y$-coordinate of $p_i$. Among them, we choose the leftmost block with a horizontal step at the top. If there is no such block, choose the leftmost block with a peak at the maximum $y$-coordinate. After a simple modification to the subpaths in \eqref{eq.interchange}, the height of the paths does not change: $\mathsf{ht}(p^{\prime}_i) = \mathsf{ht}(p_i)$. An example of the shape algorithm is illustrated below.

\begin{example}\label{ex.shape}
    Let $\overline{p}=p$ be the Motzkin path shown in Figure \ref{fig.Motzkin}. We demonstrate the whole process of applying the shape algorithm to $p$ as follows:
    \begin{itemize}
        \item Decompose $\bar{p}=p_1 \oplus p_2$. Since $p_1$ is not a unit Motzkin path, we decompose it as $p_1 = U \oplus q_{1} \oplus q_{2} \oplus D$ (see Figure \ref{fig.shapealgo1}). Notice that $q_{2}$ has the maximum height; the shape algorithm (2) yields the Motzkin path shown in Figure \ref{fig.shapealgo2}. Move back to $(1)$.
    \begin{figure}[H]
        \centering 
        \begin{tikzpicture}[scale=0.6]
            \draw[thick, black]
                (0,0) -- (1,1);
            \draw[thick, red]
                (1,1) -- (2,2) 
                      -- (3,1);
            \draw[thick, blue]
                (3,1) -- (4,2)
                      -- (5,2)
                      -- (6,1);
            \draw[thick, black]
                (6,1) -- (7,0)
                      -- (8,1)
                      -- (9,1)
                      -- (10,2)
                      -- (11,3)
                      -- (12,3)
                      -- (13,2)
                      -- (14,1)
                      -- (15,2)
                      -- (16,2)
                      -- (17,3)
                      -- (18,2)
                      -- (19,1)
                      -- (20,0);
            \foreach \x/\y in {0/0, 1/1, 2/2, 3/1, 4/2, 5/2, 6/1, 7/0, 8/1, 9/1, 10/2, 11/3, 12/3, 13/2, 14/1, 15/2, 16/2, 17/3, 18/2, 19/1, 20/0}{
                \filldraw[black] (\x,\y) circle (2pt);
            }
            \node[below] at (3.5, 0) {$p_1$};
            \node[below] at (13.5, 0) {$p_2$};
            \node[below] at (2, 1) {$q_{1}$};
            \node[below] at (4.5, 1) {$q_{2}$};
        \end{tikzpicture}
        \caption{The first block decomposition of $\bar{p}$.}
        \label{fig.shapealgo1}
    \end{figure}
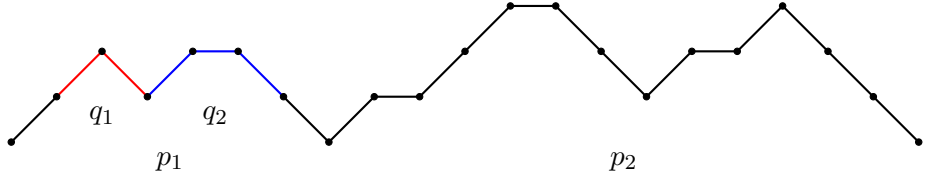

    \item Decompose $\overline{p} = p_1 \oplus p_2 \oplus p_3$. Since $p_1$ and $p_2$ are unit Motzkin paths, we decompose $p_3 = U \oplus q_{1} \oplus q_{2} \oplus q_{3} \oplus D$ (see Figure \ref{fig.shapealgo2}). Notice that $q_{2}$ has the maximum height; the shape algorithm (2) leads to the Motzkin path shown in Figure \ref{fig.shapealgo3}. Move back to (1).  
    \begin{figure}[H]
        \centering
        \begin{tikzpicture}[scale=0.6]
            \draw[thick, black]
                (0,0) -- (1,1)
                      -- (2,0) 
                      -- (3,1)
                      -- (4,2)
                      -- (5,2)
                      -- (6,1)
                      -- (7,0)
                      -- (8,1);
            \draw[thick, red]
                (8,1) -- (9,1);
            \draw[thick, blue]
                (9,1) -- (10,2)
                      -- (11,3)
                      -- (12,3)
                      -- (13,2)
                      -- (14,1);
            \draw[thick, red]
                (14,1) -- (15,2)
                      -- (16,2)
                      -- (17,3)
                      -- (18,2)
                      -- (19,1);
            \draw[thick, black]
                (19,1) -- (20,0);
            \foreach \x/\y in {0/0, 1/1, 2/0, 3/1, 4/2, 5/2, 6/1, 7/0, 8/1, 9/1, 10/2, 11/3, 12/3, 13/2, 14/1, 15/2, 16/2, 17/3, 18/2, 19/1, 20/0}{
                \filldraw[black] (\x,\y) circle (2pt);
            }
            \node[below] at (13.5, 0) {$p_3$};
            \node[below] at (1, 0) {$p_{1}$};
            \node[below] at (4.5, 0) {$p_{2}$};
            \node[below] at (8.5, 1) {$q_{1}$};
            \node[below] at (11.5, 1) {$q_{2}$};
            \node[below] at (16.5, 1) {$q_{3}$};
        \end{tikzpicture}
        \caption{The result of the first iteration and the second block decomposition of $\bar{p}$.}
        \label{fig.shapealgo2}
    \end{figure}
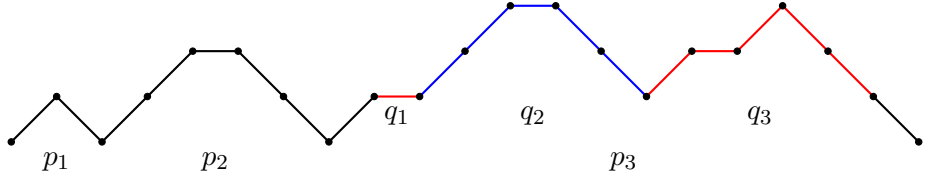

    \item Decompose $\overline{p} = p_1 \oplus \cdots \oplus p_5$. Since $p_1,p_2,p_3,p_4$ are unit Motzkin paths, we decompose $p_5 = U \oplus q_{1} \oplus q_{2} \oplus D$ (see Figure \ref{fig.shapealgo3}). Notice that $q_{2}$ has the maximum height; the shape algorithm (2) leads to the Motzkin path shown in Figure \ref{fig.shapealgo4}. Move back to (1).
    \begin{figure}[H]
        \centering
        \begin{tikzpicture}[scale=0.6]
            \draw[thick, black]
                (0,0) -- (1,1)
                      -- (2,0) 
                      -- (3,1)
                      -- (4,2)
                      -- (5,2)
                      -- (6,1)
                      -- (7,0)
                      -- (8,0)
                      -- (9,1)
                      -- (10,2)
                      -- (11,3)
                      -- (12,3)
                      -- (13,2)
                      -- (14,1)
                      -- (15,0)
                      -- (16,1);
            \draw[thick, red]
                (16,1)-- (17,1);
            \draw[thick, blue]
                (17,1)-- (18,2)
                      -- (19,1);
            \draw[thick, black]
                (19,1)-- (20,0);
    
            \foreach \x/\y in {0/0, 1/1, 2/0, 3/1, 4/2, 5/2, 6/1, 7/0, 8/0, 9/1, 10/2, 11/3, 12/3, 13/2, 14/1, 15/0, 16/1, 17/1, 18/2, 19/1, 20/0}{
                \filldraw[black] (\x,\y) circle (2pt);
            }
            
            \node[below] at (1, 0) {$p_{1}$};
            \node[below] at (4.5, 0) {$p_{2}$};
            \node[below] at (7.5, 0) {$p_{3}$};
            \node[below] at (11.5, 0) {$p_{4}$};
            \node[below] at (17.5, 0) {$p_{5}$};
            \node[below] at (16.5, 1) {$q_{1}$};
            \node[below] at (18, 1) {$q_{2}$};
        \end{tikzpicture}
        \caption{The result of the second iteration and the third block decomposition of $\bar{p}$.}
        \label{fig.shapealgo3}
    \end{figure}
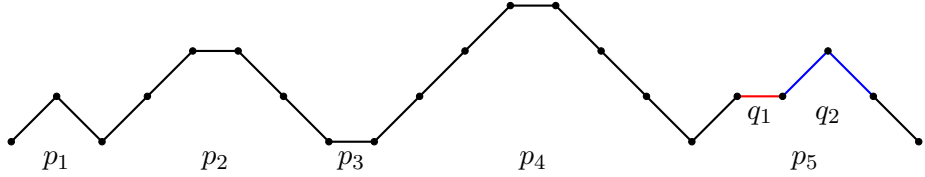

    \item Decompose $\overline{p} = p_1 \oplus \cdots \oplus p_6$. Now, each $p_i$ is a unit Motzkin path; we move to $(3)$. Reordering the length of the $p_i$'s in decreasing order, we obtain $\mu=(7,5,4,2,1,1)$ and $\theta(p)=(6,4,3,3,2,1,1)$. As mentioned in Example \ref{ex.RS}, the Motzkin path $p$ corresponds to the Richardson tableau in Figure \ref{fig.Richardsonexample} whose shape is exactly $\theta(p)$.
    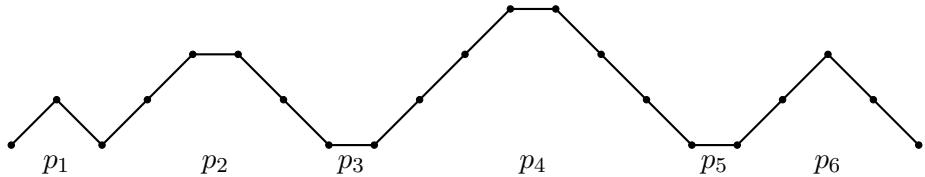
\begin{figure}[H]
        \centering
        \begin{tikzpicture}[scale=0.6]
            \draw[thick, black]
                (0,0) -- (1,1)
                      -- (2,0) 
                      -- (3,1)
                      -- (4,2)
                      -- (5,2)
                      -- (6,1)
                      -- (7,0)
                      -- (8,0)
                      -- (9,1)
                      -- (10,2)
                      -- (11,3)
                      -- (12,3)
                      -- (13,2)
                      -- (14,1)
                      -- (15,0)
                      -- (16,0)
                      -- (17,1)
                      -- (18,2)
                      -- (19,1)
                      -- (20,0);
    
            \foreach \x/\y in {0/0, 1/1, 2/0, 3/1, 4/2, 5/2, 6/1, 7/0, 8/0, 9/1, 10/2, 11/3, 12/3, 13/2, 14/1, 15/0, 16/0, 17/1, 18/2, 19/1, 20/0}{
                \filldraw[black] (\x,\y) circle (2pt);
            }
            
            \node[below] at (1, 0) {$p_{1}$};
            \node[below] at (4.5, 0) {$p_{2}$};
            \node[below] at (7.5, 0) {$p_{3}$};
            \node[below] at (11.5, 0) {$p_{4}$};
            \node[below] at (15.5, 0) {$p_{5}$};
            \node[below] at (18, 0) {$p_{6}$};
        \end{tikzpicture}
    
        \caption{The result of the third iteration.}
        \label{fig.shapealgo4}
    \end{figure}
    \end{itemize}
\end{example}

We now explain why the shape algorithm terminates in finitely many steps. The key operation of the algorithm is the modification \eqref{eq.interchange} of subpaths $p_i$ to obtain $p_i^{\prime}$; one observes that this creates two kinds of subpaths: (a) the prime Motzkin path $U^m \oplus q_{t} \oplus D^m$ with the same height as $p_i$ but shorter length, and (b) prime Motzkin paths $q_{j}$, $j \neq t$ of shorter lengths and heights. Each iteration of (2) in the shape algorithm strictly decreases the length of the highest non-unit Motzkin paths until it reaches a unit Motzkin path. Since our Motzkin paths are of finite length, this algorithm will terminate in finitely many steps. Next, we show in the following lemma that the algorithm indeed gives the desired shape of Richardson tableaux. 

We first point out a well-known result of the RS algorithm (see \cite[Chapter 3.3]{Sagan01} and references therein) and an observation.

\begin{lemma}\label{lem.well-known}
    Given a noncrossing involution $\sigma$, let $p$ and $T$ be the corresponding Motzkin path and Richardson tableau of $\sigma$, respectively. If the length of the longest decreasing subsequence of $\sigma$ is $\ell$ then the following properties hold.
    \begin{enumerate}
        \item The Richardson tableau $T$ has $\ell$ rows.
        \item The height of the Motzkin path $p$ is $\ell/2$.
    \end{enumerate}
\end{lemma}

\begin{lemma}\label{lem.shapeinvariant}
    Given a Motzkin path $p$, let $p_{\mathsf{min}}$ be the resulting Motzkin path of $p$ under the shape algorithm. Then the corresponding Richardson tableaux of $p$ and $p_{\mathsf{min}}$ have the same shape.
\end{lemma}
\begin{proof}  
    Consider a prime Motzkin path $p$ of length $n$. Let $T$ be the corresponding Richardson tableau of $p$, and let $\lambda$ be the shape of $T$. Since $p$ is prime, the arc $(1,n)$ is in the corresponding noncrossing involution of $p$. Thus, the entries $1$ and $n$ appear in the first column of $T$. Suppose $T$ contains $\ell$ rows. Then $\mathsf{ht}(p)=\ell/2$. The effect of appending a $U$ step at the beginning and a $D$ step at the end of $p$ is equivalent to adding the arc $(1,n+2)$ and increasing the original indices by $1$ in its corresponding noncrossing involution. Under the RS algorithm, the resulting Richardson tableau is obtained from $T$ by (1) increasing each entry by $1$, (2) shifting the entries in the first column $1$ unit downward, and (3) filling the top and bottom boxes in the first column with $1$ and $n+2$. So, the shape of the resulting tableau is $\lambda+(0^{\ell},1,1)$. 

    Conversely, if we remove the first $U$ step and the last $D$ step from $p$, then this is equivalent to deleting the $(1,n)$ arc and decreasing the original indices by $1$ in the corresponding noncrossing involution. Similarly, the resulting Richardson tableau is obtained from $T$ by (1) deleting the entries $1$ and $n$ (from the first column), (2) decreasing each entry by $1$, and (3) shifting the entries in the first column $1$ unit upward.

    Now, we show that the modification \eqref{eq.interchange} does not change the shape of the corresponding tableaux. Let $p=p_1\oplus \cdots \oplus p_s$ be the block decomposition of $p$. For some $i$, suppose the block $p_i = U^m \oplus q_1 \oplus q_2 \oplus \cdots \oplus q_k \oplus D^m$ for some $m\geq 1$ and each $q_j$ is a block with the base line $y=m$. Let $T_i$ be the corresponding Richardson tableau of $p_i$.  Suppose $T_i$ contains $\ell$ rows and the shape of the corresponding Richardson tableau of each $q_j$ is $\lambda^{(j)}$, for $1 \leq j \leq k$. Then $\mathsf{ht}(p_i) = \ell/2$. By the discussion in Section \ref{sec.preprime}, the shape of the Richardson tableau corresponding to $q_1 \oplus \cdots \oplus q_k$ is given by $\lambda^{(1)} + \cdots + \lambda^{(k)}$ (entrywise addition of partitions). By the above discussion, the shape of the Richardson tableau corresponding to $p_i$ is then given by
    \begin{equation}\label{eq.modify1}
        \lambda^{(1)} + \cdots + \lambda^{(k)} + (0^{\ell}, 1^{2m}).
    \end{equation}
    
    On the other hand, let $t$ be the smallest index such that for all $1 \leq j \leq k$, $\mathsf{ht}(q_j) \leq \mathsf{ht}(q_t)$. Then $\mathsf{ht}(U^m \oplus q_t \oplus D^m) = \mathsf{ht}(p_i)$. The shape of the corresponding Richardson tableau of $U^m \oplus q_t\oplus D^m$ is given by $\lambda^{(t)}+(0^{\ell},1^{2m})$. Thus, the shape of the corresponding Richardson tableau of the modification of $p_i$ can be expressed as 
    \begin{equation}\label{eq.modify2}
       \lambda^{(1)} + \cdots +\lambda^{(t-1)} +  \left( \lambda^{(t)} + (0^{\ell},1^{2m}) \right) + \lambda^{(t+1)} + \cdots +\lambda^{(k)}.
    \end{equation}

    Clearly, the partitions \eqref{eq.modify1} and \eqref{eq.modify2} are identical. Each iteration of the modification of subpaths, and hence, the shape algorithm does not change the shape of its corresponding Richardson tableau. This completes the proof of Lemma \ref{lem.shapeinvariant}.
\end{proof}

From this perspective, we can partition the set of Motzkin paths based on the output of the shape algorithm. We say a Motzkin path $p$ has \textit{shape} $\lambda$ if $\theta(p) = \lambda$. Let $\mathcal{M}(\lambda)$ denote the set of Motzkin paths of shape $\lambda$. Clearly, the set $\mathcal{M}(\lambda)$ is in bijection with the set $\mathcal{R}(\lambda)$ consisting of Richardson tableaux of shape $\lambda$. If a Motzkin path $p \in \mathcal{M}(\lambda)$, then the resulting Motzkin path under the shape algorithm (3) has the property that it consists of $\lambda_i$ unit Motzkin paths of length greater than or equal to $i$. Equivalently, it consists of exactly $\lambda_i - \lambda_{i+1}$ unit Motzkin paths of length $i$. 

One of the contributions of this paper is to provide a combinatorial proof of the $q$-enumeration formula of $\mathcal{R}(\lambda)$ (see Equation \eqref{eq.Rq-analogue}); our approach is based on analyzing those Motzkin paths of shape $\lambda$, which will be presented in the next two subsections.

\subsection{The local bijection}\label{sec.localbij}

We continue with the same notation as in Section \ref{sec.shapealgo}. Given a partition $\lambda$, the idea of constructing all the elements of $\mathcal{M}(\lambda)$ is an iterated operation of ``inserting'' the unit Motzkin paths of height $h$ into each Motzkin path obtained from those of height greater than $h$. We can insert such unit Motzkin paths only in specific positions, which can be encoded by the set of weakly increasing integer sequences. Now, we formally describe the operation mentioned above, which is called the \textit{local bijection}.

Let $p$ be a Motzkin path of shape $\lambda=(\lambda_1,\dots,\lambda_\ell)$ and $\mu = \lambda^{t} = (\mu_1,\dots,\mu_{\lambda_1})$, where $\mu_{\lambda_1}>r$ and $r$ is a positive integer. Under this setting, we would like to point out that $\lambda_1 = \lambda_2 = \cdots = \lambda_t$ for some $t = \mu_{\lambda_1} > r$. Recall from Section \ref{sec.shapealgo} that a unital segment in a Motzkin path is a maximal sequence of steps that form a unit Motzkin path. Note that the endpoints of a unital segment are not necessarily on the $x$-axis. Our goal is to insert $m$ unit Motzkin paths $u^{(r)}$ into a Motzkin path $p$ satisfying $\mu_{\lambda_1}>r$. Let $\lambda^{\prime}$ be the shape of the Motzkin path obtained from $p$ by inserting $m$ copies of $u^{(r)}$. Note that we may write $\lambda^{\prime} = (\lambda_1+m,\dots,\lambda_r+m,\lambda_{r+1},\dots,\lambda_\ell)$ and $\mu^{\prime} = (\lambda^{\prime})^{t} = (\mu_1,\dots,\mu_{\lambda_1},r^m)$. 

The local bijection is a map
\begin{equation}\label{eq.localbijection}
    \psi^{(r)}: \mathcal{S}(n,m) \times \mathcal{M}(\lambda) \to \mathcal{M}(\lambda^{\prime}),
\end{equation}
where $r$ is a positive integer, $\mathcal{S}(n,m) := \{(a_1,\dots, a_m) ~|~ 0 \leq a_1 \leq \cdots \leq a_m \leq n\}$ is the set of weakly increasing integer sequences with $n= \lambda_{r+1} + \cdots + \lambda_\ell$. The following steps describe $\psi^{(r)}$.
\begin{enumerate}
    \item For each unital segment in $p$, we mark a total of $r$ points, processed from top to bottom. If a tiebreaker is needed because two candidate points lie at the same height, we mark the left one. See Example~\ref{ex.localbij} for an illustration.
    \item In the path $p$, label all unmarked points by $0,1,\dots,n$ from right to left.
    \item Given $(a_1,\dots,a_m) \in \mathcal{S}(n,m)$, for each $i=1,\dots,m$, we insert a unit Motzkin path $u^{(r)}$ into the point labeled $a_i$ so that the right endpoint of $u^{(r)}$ is labeled by $a_i$.
\end{enumerate}
If $m=0$, then $\mathcal{S}(n,m)$ contains only the empty sequence, and the map $\psi^{(r)}$ is regarded as an identity map. 
\begin{theorem}\label{thm.localbijection}
    The map $\psi^{(r)}$ is a bijection.
\end{theorem}
Before proving Theorem \ref{thm.localbijection}, we shall give an example. 

\begin{example}\label{ex.localbij}
    Given a Motzkin path $p \in \mathcal{M}(\lambda)$ shown in Figure \ref{fig.localbij2}, its shape is given by $\lambda = (4,4,4,3,1,1,1)$ and $\mu=\lambda^{t} = (7,4,4,3)$. Now, we would like to insert $3$ copies of the unit Motzkin path of length $2$ into $p$ (i.e., $r=2$, $m=3$, and $n=10$), then the shape of the resulting Motzkin path is $\lambda^{\prime} = (7,7,4,3,1,1,1)$. If we take the integer sequence $(2,2,9) \in \mathcal{S}(10,3)$ as an example, then the map $\psi^{(2)}$ is described below. 
    \begin{itemize}
        \item There are $4$ unital segments (of lengths $5,4,3,4$ from left to right, respectively) in $p$. For each unital segment, we mark $2$ points from top to bottom; see Figure \ref{fig.localbij2}. We then label all unmarked points by $0,1,\dots,10$ from right to left. 
        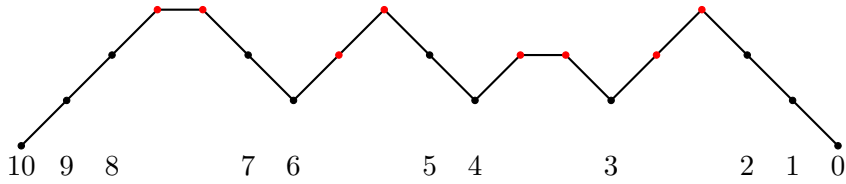
\begin{figure}[hbt!]
            \begin{tikzpicture}[scale=0.6]
                \draw[thick, black]
                    (0,0) -- (1,1)
                          -- (2,2) 
                          -- (3,3)
                          -- (4,3)
                          -- (5,2)
                          -- (6,1)
                          -- (7,2)
                          -- (8,3)
                          -- (9,2)
                          -- (10,1)
                          -- (11,2)
                          -- (12,2)
                          -- (13,1)
                          -- (14,2)
                          -- (15,3)
                          -- (16,2)
                          -- (17,1)
                          -- (18,0);
    
                \foreach \x/\y in {0/0, 1/1, 2/2, 5/2, 6/1, 9/2, 10/1, 13/1, 16/2, 17/1, 18/0}{
                    \filldraw[black] (\x,\y) circle (2pt);
                }
                \foreach \x/\y in {3/3, 4/3, 7/2, 8/3, 11/2, 12/2, 15/3, 14/2}{
                    \filldraw[red] (\x,\y) circle (2pt);
                }

                \node[below] at (0,0) {$10$};
                \node[below] at (1,0) {$9$};
                \node[below] at (2,0) {$8$};
                \node[below] at (5,0) {$7$};
                \node[below] at (6,0) {$6$};
                \node[below] at (9,0) {$5$};
                \node[below] at (10,0) {$4$};
                \node[below] at (13,0) {$3$};
                \node[below] at (16,0) {$2$};
                \node[below] at (17,0) {$1$};
                \node[below] at (18,0) {$0$};
            \end{tikzpicture}
            \caption{A Motzkin path $p$ with the marked points (in red) and the labels of the unmarked points presented in Example \ref{ex.localbij}.}
            \label{fig.localbij2}
        \end{figure}

        \item Taking $(a_1,a_2,a_3) = (2,2,9)$. We insert $u^{(2)}$ into $p$ at the lattice point labeled $a_i$ one by one. The process of each insertion is shown in Figure \ref{fig.localbij3}; the inserted unit Motzkin path is drawn in red.
        \begin{figure}[hbt!]
            \centering
            \subfigure[]{
                \begin{tikzpicture}[scale=0.6]
                    \draw[thick, black]
                        (0,0) -- (1,1)
                              -- (2,2) 
                              -- (3,3)
                              -- (4,3)
                              -- (5,2)
                              -- (6,1)
                              -- (7,2)
                              -- (8,3)
                              -- (9,2)
                              -- (10,1)
                              -- (11,2)
                              -- (12,2)
                              -- (13,1)
                              -- (14,2)
                              -- (15,3)
                              -- (16,2);
                    \draw[thick, red]
                        (16,2)-- (17,3)
                              -- (18,2);
                    \draw[thick, black]
                        (18,2)-- (19,1)
                              -- (20,0);
        
                    \foreach \x/\y in {0/0, 1/1, 2/2, 5/2, 6/1, 9/2, 10/1, 13/1, 18/2, 19/1, 20/0}{
                        \filldraw[black] (\x,\y) circle (2pt);
                    }
                    \foreach \x/\y in {3/3, 4/3, 7/2, 8/3, 11/2, 12/2, 14/2, 15/3, 16/2, 17/3}{
                        \filldraw[black] (\x,\y) circle (2pt);
                    }
    
                    \node[below] at (0,0) {$10$};
                    \node[below] at (1,0) {$9$};
                    \node[below] at (2,0) {$8$};
                    \node[below] at (5,0) {$7$};
                    \node[below] at (6,0) {$6$};
                    \node[below] at (9,0) {$5$};
                    \node[below] at (10,0) {$4$};
                    \node[below] at (13,0) {$3$};
                    \node[below] at (18,0) {$2$};
                    \node[below] at (19,0) {$1$};
                    \node[below] at (20,0) {$0$};
                \end{tikzpicture}
            }
            \vspace{1em}
            \subfigure[]{
                \begin{tikzpicture}[scale=0.6]
                    \draw[thick, black]
                        (0,0) -- (1,1)
                              -- (2,2) 
                              -- (3,3)
                              -- (4,3)
                              -- (5,2)
                              -- (6,1)
                              -- (7,2)
                              -- (8,3)
                              -- (9,2)
                              -- (10,1)
                              -- (11,2)
                              -- (12,2)
                              -- (13,1)
                              -- (14,2)
                              -- (15,3)
                              -- (16,2)
                              -- (17,3)
                              -- (18,2);
                    \draw[thick, red]
                        (18,2)-- (19,3)
                              -- (20,2);
                    \draw[thick, black]
                        (20,2)-- (21,1)
                              -- (22,0);
        
                    \foreach \x/\y in {0/0, 1/1, 2/2, 5/2, 6/1, 9/2, 10/1, 13/1, 20/2, 21/1, 22/0}{
                        \filldraw[black] (\x,\y) circle (2pt);
                    }
                    \foreach \x/\y in {3/3, 4/3, 7/2, 8/3, 11/2, 12/2, 14/2, 15/3, 16/2, 17/3, 18/2, 19/3}{
                        \filldraw[black] (\x,\y) circle (2pt);
                    }
    
                    \node[below] at (0,0) {$10$};
                    \node[below] at (1,0) {$9$};
                    \node[below] at (2,0) {$8$};
                    \node[below] at (5,0) {$7$};
                    \node[below] at (6,0) {$6$};
                    \node[below] at (9,0) {$5$};
                    \node[below] at (10,0) {$4$};
                    \node[below] at (13,0) {$3$};
                    \node[below] at (20,0) {$2$};
                    \node[below] at (21,0) {$1$};
                    \node[below] at (22,0) {$0$};
                \end{tikzpicture}
            }
            
            \vspace{1em}
            \subfigure[]{
                \begin{tikzpicture}[scale=0.6]
                    \draw[thick, black]
                        (0,0) -- (1,1);
                    \draw[thick, red]
                        (1,1) -- (2,2)
                              -- (3,1);
                    \draw[thick, black]
                        (3,1) -- (4,2)
                              -- (5,3)
                              -- (6,3)
                              -- (7,2)
                              -- (8,1)
                              -- (9,2)
                              -- (10,3)
                              -- (11,2)
                              -- (12,1)
                              -- (13,2)
                              -- (14,2)
                              -- (15,1)
                              -- (16,2)
                              -- (17,3)
                              -- (18,2)
                              -- (19,3)
                              -- (20,2)
                              -- (21,3)
                              -- (22,2)
                              -- (23,1)
                              -- (24,0);
        
                    \foreach \x/\y in {0/0, 3/1, 4/2, 7/2, 8/1, 11/2, 12/1, 15/1, 22/2, 23/1, 24/0}{
                        \filldraw[black] (\x,\y) circle (2pt);
                    }
                    \foreach \x/\y in {1/1, 2/2, 5/3, 6/3, 9/2, 10/3, 13/2, 14/2, 16/2, 17/3, 18/2, 19/3, 20/2, 21/3}{
                        \filldraw[black] (\x,\y) circle (2pt);
                    }
    
                    \node[below] at (0,0) {$10$};
                    \node[below] at (3,0) {$9$};
                    \node[below] at (4,0) {$8$};
                    \node[below] at (7,0) {$7$};
                    \node[below] at (8,0) {$6$};
                    \node[below] at (11,0) {$5$};
                    \node[below] at (12,0) {$4$};
                    \node[below] at (15,0) {$3$};
                    \node[below] at (22,0) {$2$};
                    \node[below] at (23,0) {$1$};
                    \node[below] at (24,0) {$0$};
                \end{tikzpicture}
            }        
            \caption{An illustration of inserting the unit Motzkin paths, where (a) corresponds to the first insertion at $a_1=2$, (b) corresponds to the subsequent insertion at $a_2=2$, and (c) shows the final insertion at $a_3=9$.}
            \label{fig.localbij3}
        \end{figure}
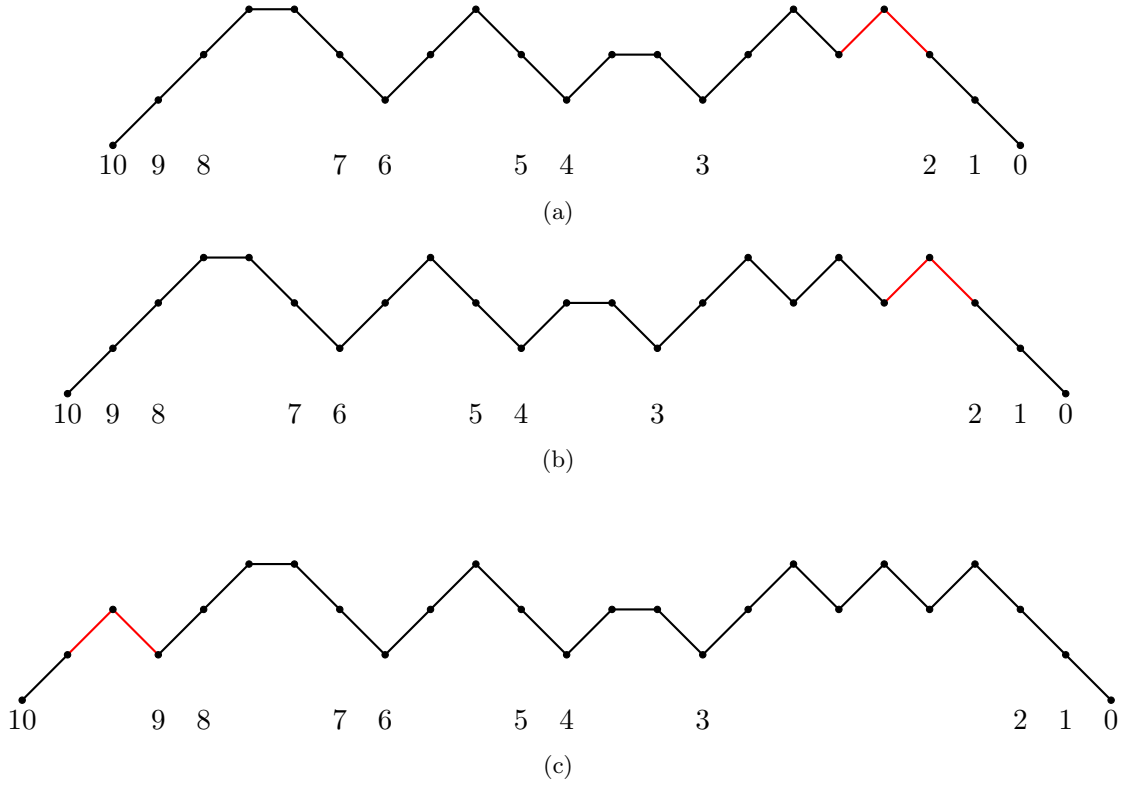
    \end{itemize}
\end{example}

The proof of Theorem \ref{thm.localbijection} follows from the following two lemmas; the first one states that the map $\psi^{(r)}$ is well-defined, and the second one tells us how to identify each piece of a unit Motzkin path and remove all of them from $p$ for the inverse map of $\psi^{(r)}$. 

\begin{lemma}\label{lem.localbij1}
    The map $\psi^{(r)}$ is well-defined.
\end{lemma}
\begin{proof}
    First, we claim that there are $\lambda_1$ unital segments in $p$. The key observation is that the modifications \eqref{eq.primedecomp} and \eqref{eq.interchange} in the shape algorithm turn each unital segment into a unit Motzkin path, without creating or deleting a unital segment. Since the number of unit Motzkin paths in $p_{\mathsf{min}}$ is $\lambda_1$, the number of unital segments in $p$ is also $\lambda_1$.

    Second, each unital segment in $p$ has length greater than $r$. Otherwise, after applying the shape algorithm, we would obtain a unit Motzkin path of length at most $r$, contradicting the assumption $\mu_{\lambda_1}>r$.

    We next count the admissible positions. For a unital segment of length greater than $r$, the marking rule selects exactly $r$ lattice points, read from top to bottom and from left to right when two points have the same height. If the $(r-1)$th and $r$th selected points lie at different heights, inserting $u^{(r)}$ at either the $r$th point or the other point at the same height in this unital segment yields the same Motzkin path. To avoid redundancy, we mark only the leftmost such point. In all other cases, a direct inspection of a pyramid and a plateau shows that these are exactly the points at which inserting $u^{(r)}$ would make the inserted path merge with the surrounding unital segment in the shape algorithm. This produces an ordered sequence $\mu^{\prime}=(\mu_1^{\prime},\dots,\mu^{\prime}_{\lambda_1},\dots)$ with $\mu^{\prime}_i > \mu_{i}$ for some $1 \leq i \leq \lambda_1$. Consequently, the resulting path fails to have the intended shape $\lambda^{\prime}$. 

    At every unmarked point, the inserted $u^{(r)}$ is separated as an additional unit Motzkin path of length $r$ in the final path $p_{\mathsf{min}}$. Since the total number of lattice points of $p$ is $|\lambda|+1$ and there are $\lambda_1$ unital segments, the number of unmarked points is
    \begin{equation*}
        |\lambda|+1-r\lambda_1
        =\lambda_{r+1}+\cdots+\lambda_\ell+1
        =n+1,
    \end{equation*}
    where we used $\lambda_1=\cdots=\lambda_t$ for some $t>r$. Hence, the labels $0,1,\dots,n$ in Step $(2)$ are well-defined.

    Finally, let $\alpha=(a_1,\dots,a_m)\in\mathcal{S}(n,m)$. Each insertion at an unmarked point contributes one additional unit Motzkin path of length $r$ under the shape algorithm, and leaves previous unit Motzkin paths unchanged. Thus, after all $m$ insertions, the multiset of lengths of the unit Motzkin paths in the resulting path is obtained from that of $p_{\mathsf{min}}$ by adjoining $m$ copies of $r$. Equivalently, the shape changes from $\lambda$ to
    \begin{equation*}
        (\lambda_1,\dots,\lambda_\ell)
        +(m^r ,0^{\ell-r})
        =\lambda^{\prime}.
    \end{equation*}
    Therefore, $\psi^{(r)}(\alpha,p)\in\mathcal{M}(\lambda^{\prime})$.
\end{proof}

For the inverse map of $\psi^{(r)}$, we use the following observation. It gives a canonical way to identify the unital segments that have to be removed.

\begin{lemma}\label{lem.numberofunitpath}
    Let $\lambda,\lambda^{\prime},r$, and $m$ be as in \eqref{eq.localbijection}, and let $p^{\prime}\in\mathcal{M}(\lambda^{\prime})$. Run the shape algorithm on $p^{\prime}$ while keeping track of the original steps. The $m$ unit Motzkin paths of length $r$ in $p^{\prime}_{\mathsf{min}}$ trace back to $m$ pairwise disjoint unital segments of length $r$ of $p^{\prime}$. Deleting these $m$ subpaths gives a Motzkin path $p\in\mathcal{M}(\lambda)$, and the deletion positions are those unmarked points of $p$ in the construction of $\psi^{(r)}$.
\end{lemma}

\begin{proof}
    Since $(\lambda^{\prime})^t$ is obtained from $\lambda^t$ by adjoining exactly $m$ parts equal to $r$, the canonical path $p^{\prime}_{\mathsf{min}}$ contains exactly $m$ unit Motzkin paths of length $r$, while all its other unit Motzkin paths have length greater than $r$. We keep these $m$ components distinguished and trace their steps backward through the recorded modifications of the shape algorithm.

    Each modification \eqref{eq.interchange} only moves prime components and wraps the selected component by the same number of $U$ and $D$ steps. If a distinguished $u^{(r)}$ were selected and wrapped, it would become a unit path of length strictly greater than $r$ and would remain unit thereafter. Hence, a distinguished component of length $r$ in the final canonical path is never changed internally; tracing backward gives an intact copy of $u^{(r)}$ in the original path $p^{\prime}$. The $m$ traced copies are pairwise disjoint. They are unital segments, since otherwise one of them would be contained in a larger unital segment and could not survive as a separate length-$r$ component in $p^{\prime}_{\mathsf{min}}$.

    Delete the $m$ traced copies from $p^{\prime}$. Deletion commutes with every recorded modification of the shape algorithm because the distinguished copies are moved only as whole subpaths. Therefore, the canonical path of the resulting path $p$ is obtained from $p^{\prime}_{\mathsf{min}}$ by deleting its $m$ unital segments of length $r$. The shape is exactly $\lambda$, so $p\in\mathcal{M}(\lambda)$.

    It remains to check the deletion positions. If one of these positions were marked in $p$, then, by the local inspection used in the proof of Lemma \ref{lem.localbij1}, reinserting $u^{(r)}$ there would merge it with the surrounding unital segment under the shape algorithm; it would not survive as a separate length-$r$ component. This contradicts the way that the distinguished components were traced from $p^{\prime}_{\mathsf{min}}$. Hence, every deletion position is unmarked.
\end{proof}

\begin{proof}[Proof of Theorem \ref{thm.localbijection}]
    Lemma \ref{lem.localbij1} shows that $\psi^{(r)}$ is well-defined. Given $p^{\prime}\in\mathcal{M}(\lambda^{\prime})$, apply Lemma \ref{lem.numberofunitpath} and delete the $m$ distinguished copies of $u^{(r)}$. This gives a unique path $p\in\mathcal{M}(\lambda)$ together with $m$ unmarked deletion positions. Label the unmarked points of $p$ by $0,1,\dots,n$ from right to left as in Step $(2)$, and let $a_1\leq\cdots\leq a_m$ be the labels of the deletion positions, with repetitions when several copies were deleted from the same position. Then $\alpha=(a_1,\dots,a_m)\in\mathcal{S}(n,m)$ and, by construction,
    \[
        \psi^{(r)}(\alpha,p)=p^{\prime}.
    \]
    The shape algorithm and the tracing procedure are deterministic, so the recovered pair $(\alpha,p)$ is unique. Thus, $\psi^{(r)}$ is a bijection.
\end{proof}

\subsection{A combinatorial proof of Theorem \ref{thm.KP25-q} and Corollary \ref{cor.qcomajor}}\label{sec.pfq-analogue}

In this subsection, we investigate how the major and comajor statistics of Motzkin paths change under the local bijection $\psi^{(r)}$. This is the key part of our combinatorial proof of the $q$-enumeration formula for the Richardson tableaux of a given shape. Recall that a weak valley (resp., weak peak) of a Motzkin path $p$ is a position at which one of the patterns $HH, DU, DH$, or $HU$ (resp., $UD, UU, DD, UH$, or $HD$) occurs. The comajor (resp., major) statistic of $p$ is the sum of all its weak valleys (resp., weak peaks); see Section \ref{sec.prestat}. For convenience, we define the following notations: for a Motzkin path $p$, $\mathsf{len}(p)$ denotes the length of $p$; for a sequence $\alpha = (a_1,\dots,a_m)$, $\mathsf{add}(\alpha) = a_1 + \cdots + a_m$ denotes the sum of all the terms in $\alpha$. We note that $\mathsf{add}(\alpha) = 0$ if $\alpha$ is the empty sequence. We have the following lemma.
\begin{lemma}\label{lem.statlocalbij}
    Let $\lambda,\lambda^{\prime},r,n$, and $m$ be as in \eqref{eq.localbijection}. Given $p \in \mathcal{M}(\lambda)$, $\alpha \in \mathcal{S}(n,m)$, and $p^\prime = \psi^{(r)}(\alpha, p)$. Then
    \begin{align}
        \mathsf{maj}(p^{\prime}) & = \mathsf{maj}(p) + \mathsf{add}(\alpha) + m(r-1)\cdot \left( \mathsf{len}(p) + \frac{rm}{2} \right), \label{eq.bijmajor} \\ 
        \mathsf{comaj}(p^\prime) & = \mathsf{comaj}(p) -\mathsf{add}(\alpha) + m \cdot \mathsf{len}(p) + r\binom{m}{2}. \label{eq.bijcomajor} 
    \end{align}
\end{lemma}    
\begin{proof}
    In the case of $m=0$, the bijection $\psi^{(r)}$ is an identity map and the set $\mathcal{S}(n,m)$ contains only the empty sequence. Here, $\mathsf{add}(\alpha) = 0$. It is obvious that both \eqref{eq.bijmajor} and \eqref{eq.bijcomajor} hold. 
    
    Now, we assume $m>0$. Let $z_i$ be the subpath of $p$ between the unmarked points labeled $i$ and $i-1$. We write $\widetilde{p_i}$ for the Motzkin path obtained from $p$ by inserting $u^{(r)}$ at the unmarked point of $p$ labeled $i$. Clearly, we have 
    \begin{align}
        p & = z_n \oplus \cdots \oplus z_1,\\
        \widetilde{p_i} & = z_n \oplus \cdots \oplus z_{i+1} \oplus u^{(r)} \oplus z_{i} \oplus \cdots \oplus z_1. \label{eq.insertconcatenation}
    \end{align}   
    For any Motzkin path $p$, it is obvious that
    \begin{equation}\label{eq.majcomaj}
        \mathsf{maj}(p) + \mathsf{comaj}(p) = \binom{\mathsf{len}(p)}{2}.
    \end{equation}
    We assume that when inserting $u^{(r)}$ into the unmarked point of $p$ labeled $i$, the comajor statistic is increased by $d(i)$, where $0 \leq i \leq n$. In other words, $d(i) = \mathsf{comaj}(\widetilde{p_i}) - \mathsf{comaj}(p)$. Due to \eqref{eq.majcomaj}, the major statistic is then increased by
    \begin{align}
        \mathsf{maj}(\widetilde{p_i}) - \mathsf{maj}(p) & = \binom{\mathsf{len}(p)+r}{2} - \mathsf{comaj}(\widetilde{p_i})  - \mathsf{maj}(p) \nonumber \\
        & = \binom{\mathsf{len}(p)+r}{2} - (\mathsf{comaj}(p) + d(i)) - \mathsf{maj}(p) \nonumber \\
        & = \binom{\mathsf{len}(p)+r}{2} - \binom{\mathsf{len}(p)}{2} - d(i) \nonumber \\
        & = r \cdot \mathsf{len}(p) + \binom{r}{2} -d(i). \label{eq.majincrement} 
    \end{align}
    
    Next, we claim that the quantity $d(i) = \mathsf{len}(p) - i$. When $i=0$, $u^{(r)}$ is inserted at the end of $p$; it is clear that this creates one weak valley and the comajor statistic is increased by $\mathsf{len}(p)$, so $d(0) = \mathsf{len}(p)$. When $i>0$, an important observation is that inserting $u^{(r)}$ at the point labeled $i+1$ can be viewed as interchanging $z_{i+1}$ and $u^{(r)}$ in \eqref{eq.insertconcatenation}. It then suffices to show that the above interchanging operation decreases the comajor statistic by $1$. 

    By the marking rule in the description of $\psi^{(r)}$, each subpath $z_i$ must be one of the forms $D$, $U$, $u^{(r+1)}$, or $U \oplus u^{(r)}$. If $z_{i+1} = D$, then the step immediately to the left of $z_{i+1}$ is a $D$ step or an $H$ step (the latter only happens when $r=1$); otherwise, the point labeled $i+1$ would form a peak $UD$, which contradicts the choice of our marked points. See Figure \ref{fig.swap} (from $\widetilde{p_0}$ to $\widetilde{p_1}$). By similar reasoning, if $z_{i+1} = U$, then the step immediately to the right of $z_{i+1}$ is a $U$ step or an $H$ step (the latter only happens when $r=1$). For these two cases, after interchanging $z_{i+1}$ and $u^{(r)}$, the weak-valley index decreases by $1$, so does the comajor statistic. In the cases of $z_{i+1} = u^{(r+1)}$ and $z_{i+1} = U \oplus u^{(r)}$, it is easy to see that interchanging $z_{i+1}$ and $u^{(r)}$ decreases the weak-valley index by $1$, and hence decreases the comajor statistic by $1$. See Figure \ref{fig.swap} (from $\widetilde{p_1}$ to $\widetilde{p_2}$ and from $\widetilde{p_2}$ to $\widetilde{p_3}$). Therefore, $d(i) = \mathsf{len}(p) - i$.
    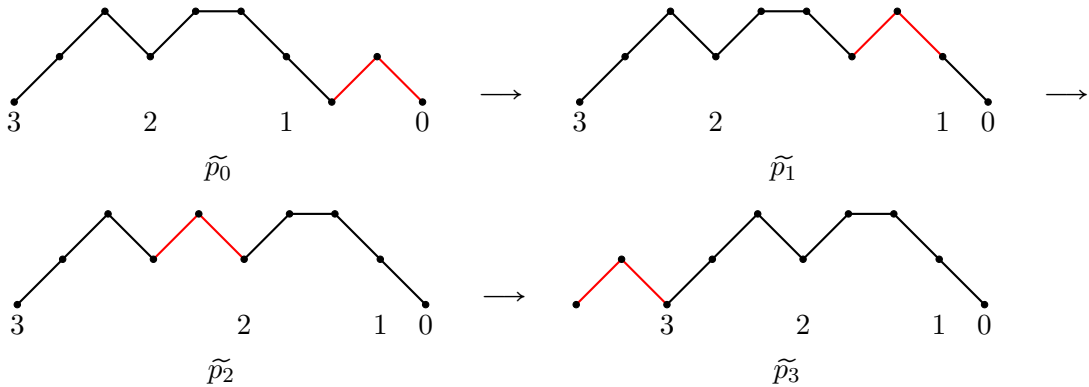
\begin{figure}[hbt!]
        \centering
        \subfigure{
        \begin{minipage}{0.4\textwidth}
            \centering
            \begin{tikzpicture}[scale=0.6]
                \draw[thick, black]
                    (0,0) -- (1,1)
                          -- (2,2) 
                          -- (3,1)
                          -- (4,2)
                          -- (5,2)
                          -- (6,1)
                          -- (7,0);
                \draw[thick, red]
                    (7,0) -- (8,1)
                          -- (9,0);
    
                \foreach \x/\y in {0/0, 1/1, 2/2, 3/1, 4/2, 5/2, 6/1, 7/0, 8/1, 9/0}{
                    \filldraw[black] (\x,\y) circle (2pt);
                }
                
                \node[below] at (0,0) {$3$};
                \node[below] at (3,0) {$2$};
                \node[below] at (6,0) {$1$};
                \node[below] at (9,0) {$0$};
                
            \end{tikzpicture}

            $\widetilde{p_0}$
        \end{minipage}
        $\longrightarrow$
        \begin{minipage}{0.4\textwidth}
            \centering
            \begin{tikzpicture}[scale=0.6]
                \draw[thick, black]
                    (0,0) -- (1,1)
                          -- (2,2) 
                          -- (3,1)
                          -- (4,2)
                          -- (5,2)
                          -- (6,1);
                \draw[thick, red]
                    (6,1) -- (7,2)
                          -- (8,1);
                \draw[thick, black]
                    (8,1) -- (9,0);
    
                \foreach \x/\y in {0/0, 1/1, 2/2, 3/1, 4/2, 5/2, 6/1, 7/2, 8/1, 9/0}{
                    \filldraw[black] (\x,\y) circle (2pt);
                }    
                
                \node[below] at (0,0) {$3$};
                \node[below] at (3,0) {$2$};
                \node[below] at (8,0) {$1$};
                \node[below] at (9,0) {$0$};
            \end{tikzpicture}

            $\widetilde{p_1}$
            
        \end{minipage}
        $\longrightarrow$
        }
        \subfigure{
        \begin{minipage}{0.4\textwidth}
            \centering
            \begin{tikzpicture}[scale=0.6]
                \draw[thick, black]
                    (0,0) -- (1,1)
                          -- (2,2) 
                          -- (3,1);
                \draw[thick, red]
                    (3,1) -- (4,2)
                          -- (5,1);
                \draw[thick, black]
                    (5,1) -- (6,2)
                          -- (7,2)
                          -- (8,1)
                          -- (9,0);
    
                \foreach \x/\y in {0/0, 1/1, 2/2, 3/1, 4/2, 5/1, 6/2, 7/2, 8/1, 9/0}{
                    \filldraw[black] (\x,\y) circle (2pt);
                }
                
                \node[below] at (0,0) {$3$};
                \node[below] at (5,0) {$2$};
                \node[below] at (8,0) {$1$};
                \node[below] at (9,0) {$0$};
            \end{tikzpicture}

            $\widetilde{p_2}$
        \end{minipage}
        $\longrightarrow$
        \begin{minipage}{0.4\textwidth}
            \centering
            \begin{tikzpicture}[scale=0.6]
                \draw[thick, red]
                    (0,0) -- (1,1)
                          -- (2,0);
                \draw[thick, black]
                    (2,0) -- (3,1)
                          -- (4,2) 
                          -- (5,1)
                          -- (6,2)
                          -- (7,2)
                          -- (8,1)
                          -- (9,0);
    
                \foreach \x/\y in {0/0, 1/1, 2/0, 3/1, 4/2, 5/1, 6/2, 7/2, 8/1, 9/0}{
                    \filldraw[black] (\x,\y) circle (2pt);
                }
                
                \node[below] at (2,0) {$3$};
                \node[below] at (5,0) {$2$};
                \node[below] at (8,0) {$1$};
                \node[below] at (9,0) {$0$};
            \end{tikzpicture}

            $\widetilde{p_3}$
        \end{minipage}
        \hspace{5mm}
        }
        \caption{An illustration of obtaining $\widetilde{p_i}$ by interchanging $z_i$ and $u^{(2)}$.}
        \label{fig.swap}
    \end{figure}

    Combining the above discussion and \eqref{eq.majincrement}, after we insert the $j$th $u^{(r)}$ into $p$ at the point labeled by $a_j$, the major statistic of the resulting Motzkin path increases by
    \begin{equation*}
        r \cdot (\mathsf{len}(p)+r(j-1) )+ \binom{r}{2} - (\mathsf{len}(p) + r(j-1) - a_j) = a_j + (r-1) \cdot (\mathsf{len}(p) + r(j-1)) + \binom{r}{2}.
    \end{equation*}
    As a consequence, when we insert $m$ copies of $u^{(r)}$ into $p$ at the labels given by $\alpha=(a_1,\dots,a_m) \in \mathcal{S}(n,m)$, respectively, the major statistic of the resulting Motzkin path $p^{\prime}$ is given by
    \begin{align*}
        \mathsf{maj}(p^{\prime}) & = \mathsf{maj}(p) + \sum_{j=1}^m \left( a_j + (r-1) \cdot (\mathsf{len}(p) + r(j-1)) + \binom{r}{2} \right)\\
            &= \mathsf{maj}(p) + \mathsf{add}(\alpha) + m(r-1) \cdot \mathsf{len}(p) + r(r-1) \binom{m}{2} + m \binom{r}{2} \\
            &= \mathsf{maj}(p) + \mathsf{add}(\alpha) + m(r-1)\cdot \left( \mathsf{len}(p) + \dfrac{rm}{2} \right).
    \end{align*}
    This proves \eqref{eq.bijmajor}.
    
    Similarly, after we insert the $j$th $u^{(r)}$ into $p$ at the point labeled by $a_j$, the comajor statistic of the resulting Motzkin path increases by
    \begin{equation*}
        \mathsf{len}(p) - a_j.
    \end{equation*}
    In general, when inserting $m$ copies of $u^{(r)}$ into $p$ at the labels given by $\alpha=(a_1,\dots,a_m) \in \mathcal{S}(n,m)$, respectively, the comajor statistic of the resulting Motzkin path $p^{\prime}$ is given by
    \begin{align*}
        \mathsf{comaj}(p^{\prime}) & = \mathsf{comaj}(p) + \sum_{j=1}^m \left( \mathsf{len}(p)+r(j-1) - a_j \right)\\
            &= \mathsf{comaj}(p) -\mathsf{add}(\alpha) + m \cdot \mathsf{len}(p) + r\binom{m}{2},
    \end{align*}
    as desired. This completes the proof of Lemma \ref{lem.statlocalbij}.
\end{proof}

Now, we are ready to prove combinatorially Theorem \ref{thm.KP25-q} and Corollary \ref{cor.qcomajor}, which state the $q$-enumeration formula for the set of Richardson tableaux of a given shape by the major and comajor statistics, respectively. 
\begin{proof}[Proof of Theorem \ref{thm.KP25-q}]
    Since the set of Richardson tableaux of shape $\lambda$ is in bijection with the set of Motzkin paths of shape $\lambda$ (Theorem \ref{thm.Guo} and Section \ref{sec.shapealgo}) and the major statistic of a Richardson tableau is the same as the major statistic of its corresponding Motzkin path (Section \ref{sec.prestat}), it is enough to evaluate the sum $\sum_{p \in \mathcal{M}(\lambda)} q^{\mathsf{maj}(p)}$.
    
    Let $\lambda=(\lambda_1,\dots,\lambda_{\ell})$ be a partition and $\mu = \lambda^t = (\mu_1,\dots, \mu_{\lambda_1})$ be the conjugate of $\lambda$. We apply a sequence of the inverse of the local bijections 
    \begin{equation}\label{eq.pfqbij}
        \left( \psi^{(\ell-1)} \right)^{-1} \circ \left( \psi^{(\ell-2)} \right)^{-1} \circ \cdots \circ \left( \psi^{(1)} \right)^{-1}
    \end{equation}
    to each Motzkin path $p \in \mathcal{M}(\lambda)$. The image of the composition \eqref{eq.pfqbij} is successively given by
    \begin{align*}
        \mathcal{M}(\lambda) &  \to 
        \mathcal{M}(\widetilde{\lambda_1}) \times \mathcal{S}(n_1, m_1) \to \mathcal{M}(\widetilde{\lambda_2}) \times \mathcal{S}(n_2, m_2) \times \mathcal{S}(n_1, m_1) \to \cdots \\
          &\to \mathcal{M}(\widetilde{\lambda_{\ell-1}}) \times \mathcal{S}(n_{\ell-1}, m_{\ell-1}) \times \cdots \times \mathcal{S}(n_1, m_1),
    \end{align*}
    where, for $1 \leq i \leq \ell-1$, $\widetilde{\lambda_i}$ is the partition whose conjugate is $(\mu_1, \dots,\mu_{\lambda_{i+1}})$, the parameter $m_i = \lambda_i - \lambda_{i+1}$ (i.e., the number of unit Motzkin paths $u^{(i)}$ in $p_\mathsf{min}$ under the shape algorithm of $p$), and the parameter $n_i = \lambda_{i+1} + \cdots + \lambda_\ell$ (i.e., the number of unmarked points of the path in the image of $\left( \psi^{(i)} \right)^{-1} \circ \cdots \circ \left( \psi^{(1)} \right)^{-1}$). In particular, if $\lambda_i = \lambda_{i+1}$, that is, $m_i=0$, then $\mathcal{S}(n_i, m_i)$ contains only the empty sequence and the bijection $\psi^{(i)}$ is the identity. For convenience, we write $p^{(i)}$ for the image of $p$ under $\left( \psi^{(i)} \right)^{-1} \circ \cdots \circ \left( \psi^{(1)} \right)^{-1}$. 
  
    In this way, we can obtain the sequences $\alpha_i \in \mathcal{S}(n_i,m_i)$ and the final Motzkin path in $\mathcal{M}(\widetilde{\lambda_{\ell-1}})$ for each $p \in \mathcal{M}(\lambda)$. Notice that the partition $\widetilde{\lambda_{\ell-1}}$ is the rectangle of size $\ell \times \lambda_{\ell}$, so the set $\mathcal{M}(\widetilde{\lambda_{\ell-1}})$ contains only one element:
    \begin{equation*}
        p^{(\ell-1)} = \underbrace{u^{(\ell)} \oplus \cdots \oplus u^{(\ell)}}_{\lambda_\ell}.
    \end{equation*}
    Since the weak valleys of $p^{(\ell-1)}$ are $\ell, 2\ell,\dots,(\lambda_\ell-1)\ell$, we have
    \begin{align}
        \mathsf{comaj}\left( p^{(\ell-1)} \right) & = \ell+2\ell + \cdots + (\lambda_\ell-1)\ell = \ell \cdot \dfrac{\lambda_\ell(\lambda_\ell - 1)}{2}, \label{eq.comajofp0} \\
        \mathsf{maj}\left( p^{(\ell-1)} \right) & = 1 + 2 + \cdots + (\lambda_\ell \cdot \ell -1) - \mathsf{comaj}\left( p^{(\ell-1)} \right) = \lambda_\ell^2 \cdot \dfrac{\ell(\ell-1)}{2}. \label{eq.majofp0}
    \end{align}
    We also observe that the length of the path $p^{(i)}$ is given by
    \begin{equation}\label{eq.lengthpi}
        \mathsf{len}\left( p^{(i)} \right) = \sum_{j=i+1}^\ell\lambda_j + i\lambda_{i+1}.
    \end{equation}

    We then apply Lemma \ref{lem.statlocalbij} to each composition in \eqref{eq.pfqbij} and obtain
    \begin{equation}\label{eq.pfmajcomposition}
        \mathsf{maj}(p) = \mathsf{maj}\left( p^{(\ell-1)} \right) + \sum_{i=1}^{\ell-1} \mathsf{add}(\alpha_i) + \sum_{i=1}^{\ell-1} m_i(i-1)\left( \mathsf{len}\left( p^{(i)} \right) + \dfrac{im_i}{2} \right).
    \end{equation}
    Note that the last summation in \eqref{eq.pfmajcomposition} simplifies to
    \begin{align}
        \sum_{i=1}^{\ell-1} m_i(i-1)\bigg( \mathsf{len}\left( p^{(i)} \right) + \dfrac{im_i}{2} \bigg)
            &= \sum_{i=1}^{\ell-1} (\lambda_i - \lambda_{i+1})(i-1) \bigg( \sum_{j=i+1}^\ell\lambda_j + \dfrac{i}{2} (\lambda_{i+1} +\lambda_i) \bigg) \nonumber \\
            &= \sum_{i=1}^{\ell-1}(i-1) \bigg( \sum_{j=i+1}^{\ell}\lambda_i\lambda_j  - \sum_{j=i+1}^{\ell}\lambda_{i+1}\lambda_j +\dfrac{i}{2} (\lambda_i^2 - \lambda_{i+1}^2)  \bigg) \nonumber\\
            &= \sum_{i=2}^{\ell-1}\bigg((i-1) \sum_{j=i+1}^{\ell}\lambda_i\lambda_j  - (i-2)\sum_{j=i}^{\ell}\lambda_{i}\lambda_j +\dfrac{i(i-1)}{2} \lambda_i^2 - \dfrac{(i-1)(i-2)}{2} \lambda_{i}^2 \bigg) \nonumber\\
            &\quad - (\ell-2) \lambda_\ell^2 - \dfrac{(\ell-1)(\ell-2)}{2} \lambda_\ell^2 \nonumber\\
            &= \sum_{i=2}^{\ell-1}\bigg( \sum_{j=i+1}^{\ell}\lambda_i\lambda_j  - (i-2)\lambda_i^2 +  \dfrac{i(i-1)}{2}\lambda_i^2 - \dfrac{(i-1)(i-2)}{2}\lambda_i^2 \bigg) \nonumber\\
            &\quad -(\ell - 2) \lambda_\ell^2 - \dfrac{(\ell-1)(\ell-2)}{2} \lambda_\ell^2 \nonumber \\
            &= \sum_{i=2}^{\ell-1} \sum_{j=i}^{\ell}\lambda_i\lambda_j - \dfrac{(\ell+1)(\ell-2)}{2} \lambda_\ell^2, \label{eq.pfmajlastsum}
    \end{align}
    where the fourth equality is obtained from the telescoping technique. Plugging \eqref{eq.majofp0} and \eqref{eq.pfmajlastsum} into \eqref{eq.pfmajcomposition}, we obtain
    \begin{align}
        \mathsf{maj}(p) &= \dfrac{\ell(\ell-1)}{2} \lambda_\ell^2 + \sum_{i=1}^{\ell-1}\mathsf{add}(\alpha_i) + \sum_{i=2}^{\ell-1} \sum_{j=i}^{\ell}\lambda_i\lambda_j - \dfrac{(\ell+1)(\ell-2)}{2} \lambda_\ell^2 \nonumber \\
            &= \sum_{i=1}^{\ell-1}\mathsf{add}(\alpha_i) + \sum_{i=2}^{\ell} \sum_{j=i}^{\ell}\lambda_i\lambda_j. \label{eq.pfmajorcombine}
    \end{align}
    
    It is a well-known result (see, for example, \cite[Theorem 3.1]{Andrews76}) that the $q$-enumeration formula for the set $\mathcal{S}(n_i,m_i)$ by $\mathsf{add}$ is 
    \begin{equation} \label{eq.qadd}
        \sum_{\alpha_i \in \mathcal{S}(n_i,m_i)} q^{\mathsf{add}(\alpha_i)} = \begin{bmatrix}
            n_i+m_i \\ n_i
        \end{bmatrix}_q = \begin{bmatrix}
            \lambda_i + \lambda_{i+2}+\cdots + \lambda_{\ell} \\ \lambda_{i+1}+\cdots + \lambda_{\ell}
        \end{bmatrix}_q.
    \end{equation}
    Finally, with \eqref{eq.pfmajorcombine} and \eqref{eq.qadd}, we obtain the desired formula
    \begin{align*}
        \sum_{T \in \mathcal{R}(\lambda)} q^{\mathsf{maj}(T)} &= \sum_{p \in \mathcal{M}(\lambda)}q^{\mathsf{maj}(p)}\\
            &= q^{\sum_{2 \le i \le j \le \ell}\lambda_i \lambda_j} \prod_{i=1}^{\ell-1}  \sum_{\alpha_i \in \mathcal{S}(n_i, m_i)}q^{\mathsf{add}(\alpha_i)}\\
            &= q^{\sum_{2 \le i \le j \le \ell}\lambda_i \lambda_j} \prod_{i=1}^{\ell-1}  \begin{bmatrix}
            \lambda_i + \lambda_{i+2} + \cdots + \lambda_\ell \\
            \lambda_{i+1} + \lambda_{i+2} +\cdots + \lambda_\ell
        \end{bmatrix}_q.
    \end{align*}
    This completes the proof of Theorem \ref{thm.KP25-q}.
\end{proof}

\begin{proof}[Proof of Corollary \ref{cor.qcomajor}]
    It is equivalent to evaluate $\sum_{p \in \mathcal{M}(\lambda)} q^{\mathsf{comaj}(p)}$. By the same idea as in the proof of Theorem \ref{thm.KP25-q} and Lemma \ref{lem.statlocalbij}, we obtain
    \begin{align}
        \mathsf{comaj}(p) & = \mathsf{comaj}\left( p^{(\ell-1)} \right) - \sum_{i=1}^{\ell-1}\mathsf{add}(\alpha_i) + \sum_{i=1}^{\ell-1} \left( m_i \cdot \mathsf{len}\left( p^{(i)} \right)+ i\binom{m_i}{2} \right). \nonumber\\ 
        & = \mathsf{comaj}\left( p^{(\ell-1)} \right) + \sum_{i=1}^{\ell-1} (n_im_i - \mathsf{add}(\alpha_i))+ \sum_{i=1}^{\ell-1} \bigg(m_i \cdot \mathsf{len}\left( p^{(i)} \right)  + i\binom{m_i}{2} - n_im_i \bigg). \label{eq.pfcomaj}
    \end{align}
    The last summation in \eqref{eq.pfcomaj} simplifies to
    \begin{align}
         \sum_{i=1}^{\ell-1} \bigg(m_i \cdot \mathsf{len}\left( p^{(i)} \right)  + i\binom{m_i}{2} - n_im_i \bigg) &= \sum_{i=1}^{\ell-1} (\lambda_i-\lambda_{i+1})\left( \sum_{j=i+1}^{\ell} \lambda_j + i\lambda_{i+1} +i\frac{\lambda_i-\lambda_{i+1}-1}{2} - \sum_{j=i+1}^{\ell} \lambda_j \right) \nonumber \\
         & = \sum_{i=1}^{\ell-1}(\lambda_i-\lambda_{i+1}) \cdot \dfrac{i}{2} \cdot (\lambda_i+\lambda_{i+1}-1) \nonumber \\
         & = \dfrac{1}{2}\sum_{i=1}^{\ell-1} \left( i(\lambda_{i}^2 - \lambda_{i+1}^2) - i(\lambda_i - \lambda_{i+1}) \right) \nonumber \\
         &= \dfrac{1}{2}\sum_{i=1}^{\ell-1} (\lambda_{i}^2 - \lambda_{i}) -\dfrac{\ell-1}{2}(\lambda_\ell^2 - \lambda_\ell) \nonumber \\
         &= \dfrac{1}{2}\sum_{i=1}^{\ell-1} \lambda_{i}(\lambda_{i}-1) -\dfrac{\ell-1}{2}\lambda_\ell(\lambda_\ell-1), \label{eq.pflastcomajor}
    \end{align}
    where the fourth equality is again obtained from the telescoping technique.
    
    For any $\alpha = (a_1,\dots,a_m) \in \mathcal{S}(n,m)$, define $a_i^\prime = n-a_{m-i+1}$ and set $\alpha^{\prime} = (a_1^{\prime},\dots,a_m^{\prime})$. It is straightforward to check that $\alpha^{\prime} \in \mathcal{S}(n,m)$ and $\alpha \mapsto \alpha^{\prime}$ is a bijection. We also have
    \begin{equation}\label{eq.alphabij}
        \mathsf{add}(\alpha^\prime) = nm - \mathsf{add}(\alpha).
    \end{equation}
    Plugging \eqref{eq.comajofp0}, \eqref{eq.pflastcomajor}, and \eqref{eq.alphabij} into \eqref{eq.pfcomaj}, we then obtain
    \begin{align}\label{eq.pfcomajorcombine}
        \mathsf{comaj}(p) & = \ell \cdot \dfrac{\lambda_\ell(\lambda_\ell - 1)}{2} + \sum_{i=1}^{\ell-1} \mathsf{add}(\alpha_i^\prime) + \dfrac{1}{2}\sum_{i=1}^{\ell-1} \lambda_{i}(\lambda_{i}-1) -\dfrac{\ell-1}{2}\lambda_\ell(\lambda_\ell-1)  \nonumber \\
        & = \sum_{i=1}^{\ell-1} \mathsf{add}(\alpha_i^\prime) +\dfrac{1}{2}\sum_{i=1}^{\ell} \lambda_{i}(\lambda_{i}-1).
    \end{align}
    Therefore, by \eqref{eq.pfcomajorcombine} and \eqref{eq.qadd},
    \begin{align*}
        \sum_{T \in \mathcal{R}(\lambda)} q^{\mathsf{comaj}(T)} &= \sum_{p \in \mathcal{M}(\lambda)} q^{\mathsf{comaj}(p)}\\
            &= q^{\frac{1}{2}\sum_{i=1}^{\ell} \lambda_{i}(\lambda_{i}-1)} \prod_{i=1}^{\ell-1}  \sum_{\alpha_i \in \mathcal{S}(n_i, m_i)}q^{\mathsf{add}(\alpha_i^\prime)}\\
            &= q^{\frac{1}{2}\sum_{i=1}^\ell \lambda_i (\lambda_i-1)} \prod_{i=1}^{\ell-1}  \begin{bmatrix}
                \lambda_i + \lambda_{i+2} + \cdots + \lambda_\ell \\
                \lambda_{i+1} + \lambda_{i+2} +\cdots + \lambda_\ell
            \end{bmatrix}_q.
    \end{align*}
    This finishes the proof of Corollary \ref{cor.qcomajor}.
\end{proof}

\section{Proofs of Guo's conjecture}\label{sec.Guoconj}

In this section, we provide two proofs of Guo's conjecture (Theorem \ref{thm.Guoconj}). The first one considers generalized Motzkin paths and identifies them with specific permutations. This identification transforms the comajor statistic into the standard major statistic on permutations, allowing us to complete the proof by applying Stanley's shuffle theorem \cite{Stan72}. The second one is bijective; we can obtain the set of Motzkin paths of given length and number of horizontal steps by carefully inserting horizontal steps into the set of Dyck paths. This insertion operation is different from the local bijection presented in Section \ref{sec.localbij}; see Remark \ref{rmk.difference}.

\subsection{The first proof}\label{sec.pfGuoshuffle}

Let $a,b,c$ be nonnegative integers. We consider a more general lattice path starting from the origin and ending at $(a+b+c,a-c)$ with $a$ up steps $(1,1)$ (the $U$ step), $b$ horizontal steps $(1,0)$ (the $H$ step), and $c$ down steps $(1,-1)$ (the $D$ step). Let $\mathcal{G}(a,b,c)$ be the collection of such lattice paths. By convention, $\mathcal{G}(0,0,0)$ consists of the lattice path of length $0$. We write $\mathcal{G}^{+}(a,b,c)$ for the subset of $\mathcal{G}(a,b,c)$ consisting of lattice paths that stay weakly above the $x$-axis, and define $\mathcal{G}^{-}(a,b,c) = \mathcal{G}(a,b,c) \setminus \mathcal{G}^{+}(a,b,c)$. In particular, the set of Motzkin paths of length $n$ is a special case of $\mathcal{G}^{+}(a,b,c)$ with $a+b+c=n$ and $a=c$. 

We have the following lemma about the generating function of the generalized Motzkin paths weighted by the comajor statistic (Section \ref{sec.prestat}).
\begin{lemma}\label{lem.gmgf}
    Let $a,b,c$ be nonnegative integers. Then the generating function of $\mathcal{G}(a,b,c)$ weighted by the comajor statistic is given by
    \begin{equation}\label{eq.gmgf}
        \sum_{p \in \mathcal{G}(a,b,c)} q^{\mathsf{comaj}(p)} = q^{\binom{b}{2}} \begin{bmatrix}
            a+b+c \\ a,b,c
        \end{bmatrix}_q .
    \end{equation}
\end{lemma}

To prove Lemma \ref{lem.gmgf}, we recall the notion of a shuffle of two permutations and Stanley's shuffle theorem. Let $m$ and $n$ be two positive integers. We write $\mathfrak{B}_n$ for the set of permutations of $n$ distinct numbers, not necessarily from $1$ to $n$. Let $\pi \in \mathfrak{B}_n$ and $\sigma \in \mathfrak{B}_m$ be two disjoint permutations (numbers in $\pi$ and $\sigma$ are distinct). We say $\tau \in \mathfrak{B}_{n+m}$ is a \textit{shuffle} of $\pi$ and $\sigma$ if $\pi$ and $\sigma$ are subsequences of $\tau$. Denote by $\mathfrak{S}(\pi,\sigma)$ the set of shuffles of $\pi$ and $\sigma$. It is clear that $|\mathfrak{S}(\pi,\sigma)| = \binom{n+m}{n}$. Stanley's classical shuffle theorem \cite{Stan72} is stated below.
\begin{theorem}[\cite{Stan72}]\label{thm.Stanleyshuffle}
    Let $m$ and $n$ be two positive integers, and $\pi \in \mathfrak{B}_n$ and $\sigma \in \mathfrak{B}_m$ be two disjoint permutations. Then
    \begin{equation}
        \sum_{\tau \in \mathfrak{S}(\pi,\sigma)} q^{\mathsf{maj}(\tau)} = q^{\mathsf{maj}(\pi) + \mathsf{maj}(\sigma)} \begin{bmatrix}
            n+m \\
            m
        \end{bmatrix}_q.
    \end{equation}
\end{theorem}

Now, we are ready to prove Lemma \ref{lem.gmgf}.
\begin{proof}[Proof of Lemma \ref{lem.gmgf}]
    We define a function $f: \mathcal{G}(a,b,c) \rightarrow \mathfrak{S}_{a+b+c}$ by sending a generalized Motzkin path $p$ to a permutation as follows. From left to right, we replace each $U$ step in $p$ with $1,2,\dots,a$, replace each $H$ step in $p$ with $a+b,a+b-1,\dots,a+2,a+1$, and replace each $D$ step in $p$ with $a+b+1,a+b+2,\dots,a+b+c$. This gives a permutation $f(p)$ on $[a+b+c]$. The key observation is that $i \in [a+b+c-1]$ is a weak valley (i.e., an index of the occurrence of $HH$, $DU$, $DH$, and $HU$) of $p$ if and only if $i$ is a descent of the permutation $f(p)$. Thus, the comajor statistic of $p$ equals the usual major statistic of $f(p)$, that is, $\mathsf{comaj}(p) = \mathsf{maj}(f(p))$. So the left-hand side of \eqref{eq.gmgf} can be written as
    \begin{equation}\label{eq.gmgf2}
        \sum_{p \in \mathcal{G}(a,b,c)} q^{\mathsf{comaj}(p)} = \sum_{\pi \in f \left( \mathcal{G}(a,b,c) \right)} q^{\mathsf{maj}(\pi)}. 
    \end{equation}

    Notice that the set of permutations $f(\mathcal{G}(a,b,c))$ can be obtained from shuffling permutations in two stages: (1) shuffle the permutations $\pi_1 = 1,2,\dots,a$ and $\sigma_1 = a+b,a+b-1,\dots,a+1$; and (2) take $\pi_2 \in \mathfrak{S}(\pi_1,\sigma_1)$, then shuffle $\pi_2$ and $\sigma_2=a+b+1,a+b+2,\dots,a+b+c$. To evaluate \eqref{eq.gmgf2}, we apply Stanley's shuffle theorem (Theorem \ref{thm.Stanleyshuffle}) twice. For the first stage, we note that $\mathsf{maj}(\pi_1)=0$ and $\mathsf{maj}(\sigma_1)=\binom{b}{2}$. By Theorem \ref{thm.Stanleyshuffle}, we obtain the generating function for $\mathfrak{S}(\pi_1,\sigma_1)$,
    \begin{equation}\label{eq.gmgf3}
        \sum_{ \tau \in \mathfrak{S}(\pi_1,\sigma_1)}q^{\mathsf{maj}(\tau)} = q^{\binom{b}{2}} \begin{bmatrix}
            a+b \\ b
        \end{bmatrix}_q.
    \end{equation}
    For the second stage, note that $\mathsf{maj}(\sigma_2) = 0$. Then Theorem \ref{thm.Stanleyshuffle} implies that
        \begin{equation}\label{eq.gmgf4}
            \sum_{\tau \in \mathfrak{S}(\pi_2, \sigma_2)} q^{\mathsf{maj}(\tau)}  = q^{\mathsf{maj}(\pi_2)} \begin{bmatrix}
                a+b+c \\ c
            \end{bmatrix}_q.
        \end{equation}  

    Finally, we obtain
    \begin{align}
        \sum_{\pi \in f \left( \mathcal{G}(a,b,c) \right)} q^{\mathsf{maj}(\pi)} & = \sum_{\pi_2 \in \mathfrak{S}(\pi_1,\sigma_1)} \sum_{\tau \in \mathfrak{S}(\pi_2,\sigma_2)} q^{\mathsf{maj}(\tau)} \nonumber \\
         & = \sum_{\pi_2 \in \mathfrak{S}(\pi_1,\sigma_1)} q^{\mathsf{maj}(\pi_2)} \begin{bmatrix}
            a+b+c \\ c
        \end{bmatrix}_q \label{eq.gmgf5}\\
        & = q^{\binom{b}{2}} \begin{bmatrix}
            a+b \\ b
        \end{bmatrix}_q
        \begin{bmatrix}
            a+b+c \\ c
        \end{bmatrix}_q \label{eq.gmgf6} \\
        & = q^{\binom{b}{2}} \begin{bmatrix}
            a+b+c \\ a,b,c
        \end{bmatrix}_q, \nonumber
    \end{align}
    where Equations \eqref{eq.gmgf5} and \eqref{eq.gmgf6} follow from \eqref{eq.gmgf4} and \eqref{eq.gmgf3}, respectively. This completes the proof of Lemma \ref{lem.gmgf}.  
\end{proof}

We need the following lemma for Motzkin paths, which is similar to the one presented in \cite[Page 255]{FH85} for Dyck paths. 
\begin{lemma}\label{lem.reflection}
    For $a \geq c$, there exists a bijection $\phi:\mathcal{G}^{-}(a,b,c) \rightarrow \mathcal{G}(a+1,b,c-1)$ such that $\mathsf{comaj}(\phi(p)) = \mathsf{comaj}(p) -1$ for all $p \in \mathcal{G}^{-}(a,b,c)$.
\end{lemma}
\begin{proof}
    The bijection $\phi$ is constructed below. Given a path $p \in \mathcal{G}^{-}(a,b,c)$, find the lowest lattice point of $p$; if there is more than one such point, then we take the leftmost one and call it $v$. Let $u$ be the lattice point before $v$ on $p$. Note that the step $uv$ must be a $D$ step. We write $p_1$ (resp., $p_2$) for the subpath of $p$ before (resp., after) this $D$ step. Then the map $\phi$ sends $p_1 \oplus D \oplus p_2$ to $p_1 \oplus U \oplus p_2$. It is clear that $\phi(p) \in \mathcal{G}(a+1,b,c-1)$ and the index of the weak valley involving this $D$ step decreases by $1$ under $\phi$. Thus, $\mathsf{comaj}(\phi(p)) = \mathsf{comaj}(p) -1$. The condition $a \geq c$ implies that $v$ is not the ending point of the path, which makes the above argument valid.

    It is easy to see that $\phi$ is a bijection: we start with $p^{\prime} \in \mathcal{G}(a+1,b,c-1)$, find the lowest lattice point on $p^{\prime}$; if there is more than one such point, then we take the rightmost one and call it $v^{\prime}$. Let $u^{\prime}$ be the lattice point after $v^{\prime}$ on $p^{\prime}$. It is clear that the step $v^{\prime}u^{\prime}$ is a $U$ step. The inverse map is then given by flipping this $U$ step into a $D$ step on $p^{\prime}$. This completes the proof of Lemma \ref{lem.reflection}. 
\end{proof}

Recall that Guo's conjecture (Theorem \ref{thm.Guoconj}) provides the formula for the generating function of Richardson tableaux $R(n,k)$ of size $n$ and with $k$ odd columns; its proof is presented below.
\begin{proof}[The first proof of Theorem \ref{thm.Guoconj}.]
    Recall from Section \ref{sec.preRS} that, under the RS algorithm, the number of fixed points of an involution (equivalently, the number of $H$ steps of its corresponding Motzkin path) is the same as the number of odd columns in the insertion tableau. By Theorem \ref{thm.Guo}, the set $R(n,k)$ is in bijection with $\mathcal{I}(n,k)$ (equivalently, $\mathcal{M}(n,k)$). By Lemma \ref{lem.weakvally} and the paragraphs after Lemma \ref{lem.weakvally}, the index $i$ is an ascent of a Richardson tableau $T$ if and only if $i$ is a weak valley of the corresponding Motzkin path $p$. Thus, the comajor statistic of $T$ is the same as the comajor statistic of its corresponding Motzkin path $p$. Thus,
    \begin{equation}\label{eq.Guoconjequiv}
        \sum_{T \in \mathcal{R}(n,k)} q^{\mathsf{comaj}(T)} = \sum_{p \in \mathcal{M}(n,k)} q^{\mathsf{comaj}(p)}.
    \end{equation}

    Now, we consider the following generating function of $\mathcal{G}^{+}(a,b,c)$ for $a \geq c$,
    \begin{align}
        \sum_{p \in \mathcal{G}^{+}(a,b,c)} q^{\mathsf{comaj}(p)} & = \sum_{p \in \mathcal{G}(a,b,c)} q^{\mathsf{comaj}(p)} - \sum_{p \in \mathcal{G}^{-}(a,b,c)} q^{\mathsf{comaj}(p)}  & \nonumber \\
        & = \sum_{p \in \mathcal{G}(a,b,c)} q^{\mathsf{comaj}(p)} - \sum_{p \in \mathcal{G}(a+1,b,c-1)} q^{\mathsf{comaj}(p)+1} & \text{(by Lemma \ref{lem.reflection})} \nonumber \\
        & = q^{\binom{b}{2}} \begin{bmatrix}
            a+b+c\\a,b,c
        \end{bmatrix}_q - 
        q^{\binom{b}{2}+1} \begin{bmatrix}
            a+b+c\\a+1,b,c-1
        \end{bmatrix}_q & \text{(by Lemma \ref{lem.gmgf})} \nonumber \\
        & = q^{\binom{b}{2}} \frac{[a+b+c]_q!}{[a]_q![b]_q![c-1]_q!} \left( \frac{1}{[c]_q} - \frac{q}{[a+1]_q} \right) \nonumber \\
        & = q^{\binom{b}{2}} \frac{[a+b+c]_q!}{[b]_q! [a+c]_q!} \cdot \frac{[a+c]_q!}{[a]_q! [c]_q!} \cdot [c]_q \left( \frac{1}{[c]_q} - \frac{q}{[a+1]_q} \right). \label{eq.pfconj}
    \end{align}
    It is clear that when taking $a=c=\frac{n-k}{2}$ and $b=k$, the set $\mathcal{G}^{+}(a,b,c)$ reduces to $\mathcal{M}(n,k)$, and the second and third terms in \eqref{eq.pfconj} give the $q$-Catalan number $C_{\frac{n-k}{2}}(q)$. We then obtain the desired result. 
\end{proof}

\subsection{The second proof}\label{sec.pfGuoinsertion}

A \textit{Dyck path} of length $2n$ is a special case of a Motzkin path of length $2n$ with no horizontal steps, that is, an element in the set $\mathcal{M}(2n,0)$. For convenience, we write $\mathcal{C}(n)$ for the set of Dyck paths of length $2n$. A valley (resp., peak) of a Dyck path $\tau$ is defined as an index $i \in [2n-1]$ such that the steps $i$ and $i+1$ of $\tau$ are of the form $DU$ (resp., $UD$). The major (resp., comajor) statistic of a Dyck path is given by the sum of all its peaks (resp., valleys). Similarly to \eqref{eq.localbijection}, we define the set of strictly increasing integer sequences
\begin{equation*}
    \mathcal{T}(n,m) := \{(a_1,a_2,\dots,a_m)\mid 0 \leq a_1<a_2< \cdots <a_m<n\}.
\end{equation*}
Note that if $m=0$, then $\mathcal{T}(n,m)$ contains only the empty sequence. 

The following theorem plays an important role in proving Theorem \ref{thm.Guoconj}.
\begin{theorem}\label{thm.tuple-Dyck-to-Motzkin}
    For $0\le k \le n$ with $k \equiv n \pmod 2$, let $n=k+2\ell$. There is a bijection 
    \begin{equation}
        \zeta:\mathcal{T}(n,k) \times \mathcal{C}(\ell) \rightarrow \mathcal{M}(n,k)
    \end{equation}
    such that the map $(\alpha,\tau) \mapsto \omega$ satisfies the identity $\mathsf{comaj}(\omega)=\mathsf{add}(\alpha)+ \mathsf{comaj}(\tau)$.
\end{theorem}

Given a Dyck path $\tau \in \mathcal{C}(n)$, we associate $\tau$ with a $(2n+1)$-tuple $(b_0, b_1, \dots, b_{2n})$ of integers defined by 
\begin{equation} \label{eq.vector}
b_i := | \{ j \mid j>i \text{ and $j$ is a valley} \} |,
\end{equation}
for $0\le i\le 2n$. For example, if $\tau \in \mathcal{C}(5)$ is the Dyck path shown in Figure \ref{fig.secondproofex}, then $(b_0, b_1, \dots, b_{10})=(2,2,2,1,1,1,1,1,0,0,0)$. 

Now, we describe the map $\zeta$ and its inverse as follows. Let $n=k+2\ell$. Given an ordered pair  $(\alpha,\tau) \in \mathcal{T}(n,k) \times \mathcal{C}(\ell)$, we will construct a Motzkin path $\omega \in \mathcal{M}(n,k)$ from $(\alpha,\tau)$. If $\alpha=(0,1,\dots, n-1)$ and $\tau$ is empty, we set $\omega$ to be the path consisting of $n$ horizontal steps. If $\alpha$ is empty, then $n=2\ell$ and we set $\omega=\tau$. For $1\leq k\leq n-1$,
let $\alpha=(a_1,a_2,\dots,a_k)$ and let $\tau_i$ be the $i$th step of $\tau$. We also let $(b_0,b_1,\dots,b_{2\ell})$ be the $(2\ell+1)$-tuple associated to the Dyck path $\tau$. The Motzkin path $\omega$ is expressed as
\begin{equation} \label{eqn:expression}
\omega = \eta_0\oplus \tau_1 \oplus \eta_1 \oplus \tau_2 \oplus \eta_2 \oplus \cdots \oplus \tau_{2\ell}\oplus\eta_{2\ell},
\end{equation}
where each $\eta_j$ is a sequence of consecutive horizontal steps (possibly empty) for $0\leq j\leq 2\ell$. To obtain \eqref{eqn:expression}, we construct a sequence $\omega_0=\tau, \omega_1,\dots,\omega_k=\omega$ of Motzkin paths, where $\omega_i$ is obtained from $\omega_{i-1}$ by adding a horizontal step, $1\leq i\leq k$. This process is presented in the following algorithm.

\medskip
\noindent{\bf Algorithm A.}

We associate the sections $\eta_0,\eta_1,\dots,\eta_{2\ell}$ of each $\omega_i$ with a vector $F_i=(f_i(0), f_i(1), \dots, f_i(2\ell))$ of integers, where $f_i(j)$ keeps track of the increments to the comajor statistic when a horizontal step is inserted in the section $\eta_j$ of $\omega_i$ for $0\le j\le 2\ell$.

\begin{enumerate}
\item[(1)] The initial case is $\omega_0=\tau$, where the sections $\eta_0,\eta_1,\dots,\eta_{2\ell}$ are all empty. 
Suppose $\tau$ contains $d$ occurrences of the peak $UD$, for some $d \geq 1$. Let $\{c_1, c_2,\dots, c_d\}$ be the set of peaks of $\tau$, where $0<c_1<c_2<\cdots<c_d<2\ell$. The vector $F_0$ is given as follows. The $d$ peaks are assigned $0,1,\dots,d-1$ from right to left, i.e., $f_0(c_t)=d-t$ for $1\le t\le d$. The rest of the entries of $F_0$ are assigned $d, d+1, \dots, 2\ell$ from left to right. Notice that in terms of the tuple $(b_0, b_1, \dots, b_{2\ell})$, for $0\leq j\leq 2\ell$, we have
\begin{equation} \label{eqn:each_entry}
    f_0(j) = 
    \begin{cases}
    b_j, & \mbox{if $j\in\{c_1,c_2,\dots,c_d\}$,} \\
    j+b_j+1, &\mbox{if  $\eta_j$ is followed by an up step,} \\
    j+b_j, &\mbox{if $j\not\in\{c_1,c_2,\dots,c_d\}$ and $\eta_j$ is not followed by an up step.}
    \end{cases}
\end{equation}

\item[(2)] For $i\ge 1$, suppose the Motzkin path $\omega_{i-1}$ with the vector $F_{i-1}$ has been constructed. To construct $\omega_i$, we insert a horizontal step in one of the sections $\eta_0,\eta_1,\dots,\eta_{2\ell}$ of $\omega_{i-1}$ and obtain the vector $F_i$ from $F_{i-1}$ according to the following two cases of $a_i$ (we shall prove in Lemma \ref{lem.insertion} that $a_i$ is an entry in $F_{i-1}$).

Case 1. $a_i<d$. Then the integer $a_i$ is assigned to a peak, say $c_t$ for some $t$ (i.e., $f_{i-1}(c_t)=a_i$). We add a horizontal step in the section $\eta_{c_t}$ and obtain the vector $F_i$ by setting 
\begin{equation} \label{eqn:case-1}
    f_i(h) = \begin{cases}
    c_t+1+b_{c_t}, & \mbox{if $h=c_t$,} \\
    f_{i-1}(h)+1,  & \mbox{if $f_{i-1}(h)\geq c_t+1+b_{c_t}$,} \\
    f_{i-1}(h),     & \mbox{otherwise,}
    \end{cases}
\end{equation}
for $0\le h\le 2\ell$.

Case 2. $a_i\ge d$. Let $f_{i-1}(j)=a_i$ for some $j$. We add a horizontal step in the section $\eta_{j}$ and obtain the vector $F_i$ by setting 
\begin{equation} \label{eqn:case-2}
    f_i(h) = \begin{cases}
    f_{i-1}(h)+1,  & \mbox{if $f_{i-1}(h)\geq a_i$,} \\
    f_{i-1}(h),     & \mbox{otherwise,}
    \end{cases}
\end{equation}
for $0\le h\le 2\ell$.
\end{enumerate}

\begin{example}\label{ex.construction-Motzkin-path} 
    Let $(\alpha,\tau) \in \mathcal{T}(15,5) \times \mathcal{C}(5)$, where $\alpha=(1,2,5,12,14)$ and $\tau=UUDUUDDDUD$; see Figure \ref{fig.secondproofex}. Notice that $\mathsf{add}(\alpha)=34$ and $\mathsf{comaj}(\tau)=11$. Using Algorithm A, each stage of the Motzkin paths $\omega_0,\omega_1,\dots,\omega_5$ is shown in Figure \ref{fig.Motzkin-path-construction}. The vectors $F_0,F_1,\dots,F_5$ are listed in Table \ref{tab.step-by-step}. The resulting Motzkin path $\omega$ is given in Figure \ref{fig.secondproofex5} with $\mathsf{comaj}(\omega)=45$.
\end{example}
\begin{figure}[hbt!]
        \centering
        \subfigure[]{\label{fig.secondproofex}
            \begin{tikzpicture}[scale=0.6]
                \draw[thick, black]
                    (0,0) -- (1,1)
                          -- (2,2) 
                          -- (3,1)
                          -- (4,2)
                          -- (5,3)
                          -- (6,2)
                          -- (7,1)
                          -- (8,0)
                          -- (9,1)
                          -- (10,0);
    
                \foreach \x/\y in {0/0, 1/1, 2/2, 3/1, 4/2, 5/3, 6/2, 7/1, 8/0, 9/1, 10/0}{
                    \filldraw[black] (\x,\y) circle (2pt);
                }    
            \end{tikzpicture}
        }
        
        \vspace{1em}
        \subfigure[]{\label{fig.secondproofex1}
            \begin{tikzpicture}[scale=0.6]
                \draw[thick, black]
                    (0,0) -- (1,1)
                          -- (2,2) 
                          -- (3,1)
                          -- (4,2)
                          -- (5,3);
                \draw[thick, red]
                    (5,3) -- (6,3);
                \draw[thick, black]           
                    (6,3) -- (7,2)
                          -- (8,1)
                          -- (9,0)
                          -- (10,1)
                          -- (11,0);
    
                \foreach \x/\y in {0/0, 1/1, 2/2, 3/1, 4/2, 5/3, 6/3, 7/2, 8/1, 9/0, 10/1, 11/0}{
                    \filldraw[black] (\x,\y) circle (2pt);
                }    
            \end{tikzpicture}
        }
        
        \vspace{1em}
        \subfigure[]{\label{fig.secondproofex2}
            \begin{tikzpicture}[scale=0.6]
                \draw[thick, black]
                    (0,0) -- (1,1)
                          -- (2,2);
                \draw[thick, red]
                    (2,2) -- (3,2);          
                \draw[thick, black]
                    (3,2) -- (4,1)
                          -- (5,2)
                          -- (6,3)
                          -- (7,3)
                          -- (8,2)
                          -- (9,1)
                          -- (10,0)
                          -- (11,1)
                          -- (12,0);
    
                \foreach \x/\y in {0/0, 1/1, 2/2, 3/2, 4/1, 5/2, 6/3, 7/3, 8/2, 9/1, 10/0, 11/1, 12/0}{
                    \filldraw[black] (\x,\y) circle (2pt);
                }    
            \end{tikzpicture}
        }
        
        \vspace{1em}
        \subfigure[]{\label{fig.secondproofex3}
            \begin{tikzpicture}[scale=0.6]
                \draw[thick, black]
                    (0,0) -- (1,1)
                          -- (2,2)
                          -- (3,2);
                \draw[thick, red]
                    (3,2) -- (4,2);
                \draw[thick, black]    
                    (4,2) -- (5,1)
                          -- (6,2)
                          -- (7,3)
                          -- (8,3)
                          -- (9,2)
                          -- (10,1)
                          -- (11,0)
                          -- (12,1)
                          -- (13,0);
    
                \foreach \x/\y in {0/0, 1/1, 2/2, 3/2, 4/2, 5/1, 6/2, 7/3, 8/3, 9/2, 10/1, 11/0, 12/1, 13/0}{
                    \filldraw[black] (\x,\y) circle (2pt);
                }    
            \end{tikzpicture}
        }
        
        \vspace{1em}
        \subfigure[]{\label{fig.secondproofex4}
            \begin{tikzpicture}[scale=0.6]
                \draw[thick, black]
                    (0,0) -- (1,1)
                          -- (2,2)
                          -- (3,2)
                          -- (4,2)
                          -- (5,1)
                          -- (6,2)
                          -- (7,3)
                          -- (8,3)
                          -- (9,2)
                          -- (10,1)
                          -- (11,0);
                \draw[thick, red]          
                    (11,0) -- (12,0);
                \draw[thick, black]    
                    (12,0) -- (13,1)
                          -- (14,0);
    
                \foreach \x/\y in {0/0, 1/1, 2/2, 3/2, 4/2, 5/1, 6/2, 7/3, 8/3, 9/2, 10/1, 11/0, 12/0, 13/1, 14/0}{
                    \filldraw[black] (\x,\y) circle (2pt);
                }    
            \end{tikzpicture}
        }
        
        \vspace{1em}
        \subfigure[]{\label{fig.secondproofex5}
            \begin{tikzpicture}[scale=0.6]
                \draw[thick, black]
                    (0,0) -- (1,1)
                          -- (2,2)
                          -- (3,2)
                          -- (4,2)
                          -- (5,1)
                          -- (6,2)
                          -- (7,3)
                          -- (8,3)
                          -- (9,2)
                          -- (10,1)
                          -- (11,0)
                          -- (12,0)
                          -- (13,1)
                          -- (14,0);
                \draw[thick, red]
                    (14,0) -- (15,0);
    
                \foreach \x/\y in {0/0, 1/1, 2/2, 3/2, 4/2, 5/1, 6/2, 7/3, 8/3, 9/2, 10/1, 11/0, 12/0, 13/1, 14/0, 15/0}{
                    \filldraw[black] (\x,\y) circle (2pt);
                }    
            \end{tikzpicture}
        }
        \caption{The construction of a Motzkin path using Algorithm A in Example \ref{ex.construction-Motzkin-path}. Each newly added horizontal step is drawn in red.}
        \label{fig.Motzkin-path-construction}
\end{figure}
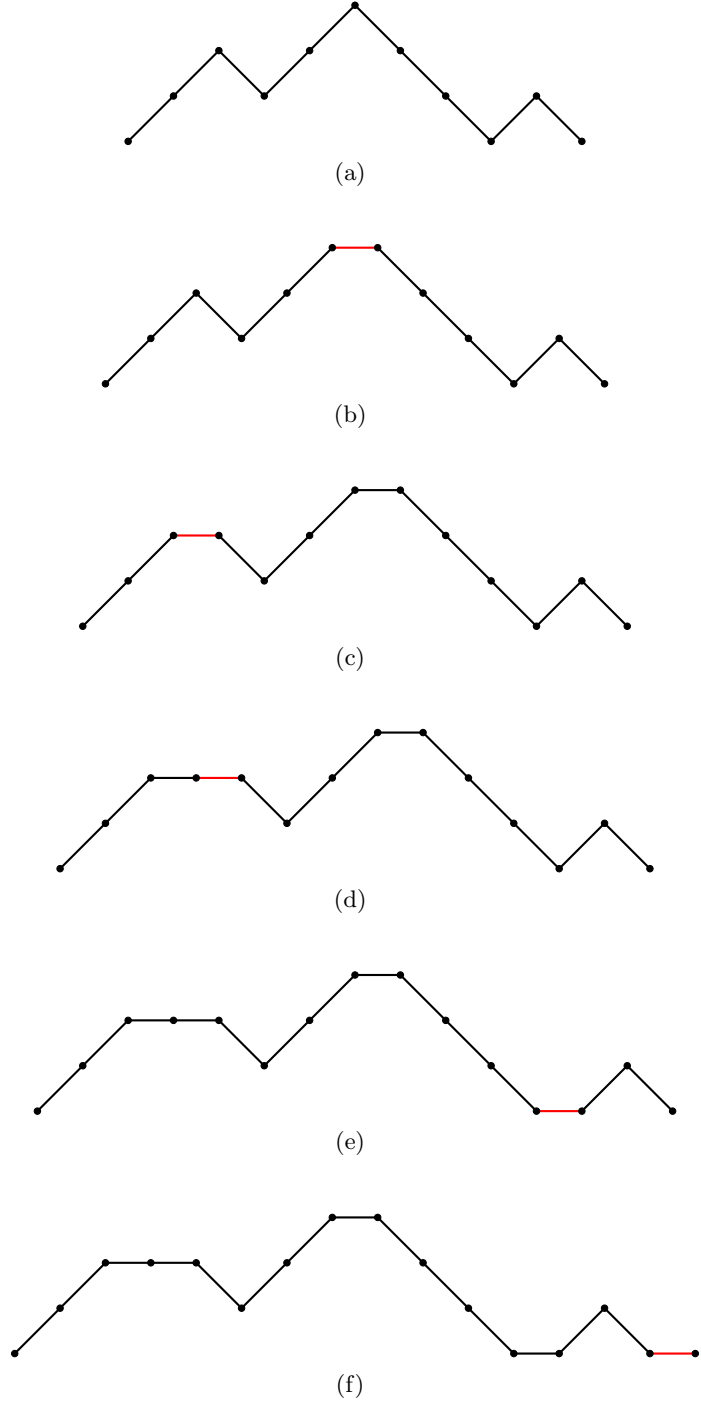

\begin{table}[hbt!]
    \centering
    {\footnotesize
    \begin{tabular}{c|c|ccccccccccc|c}
    $i$ & $a_i$     & $f_i(0)$  &  $f_i(1)$   &  $f_i(2)$ & $f_i(3)$ & $f_i(4)$ & $f_i(5)$ & $f_i(6)$ & $f_i(7)$  & $f_i(8)$ & $f_i(9)$ & $f_i(10)$  &  Figure \ref{fig.Motzkin-path-construction}\\
    \hline
    0   &   &  3  &  4   &  2  &  5  &  6  &  1  &  7  &   8  &  9  & 0  & 10  &  (a) \\
    1  &  1    &  3  &  4   &  2  &  5  &  6  &  7  &  8  &   9  &  10  & 0  & 11  &  (b) \\
    2  &  2    &  3  &  4   &  5  &  6  &  7  &  8  &  9  &   10  &  11  & 0  & 12  &  (c) \\
    3  &  5    &  3  &  4   &  6  &  7  &  8  &  9  &  10  &   11  &  12  & 0  & 13  &  (d) \\
    4  &  12   &  3  &  4   &  6  &  7  &  8  &  9  &  10  &   11  &  13  & 0  & 14  &  (e) \\
    5  &  14   &  3  &  4   &  6  &  7  &  8  &  9  &  10  &   11  &  13  & 0  & 15  &  (f)
    \end{tabular}
    }
    \vspace{2mm}
    \caption{The vectors $F_i$ in Example \ref{ex.construction-Motzkin-path}.}
    \label{tab.step-by-step}
\end{table}

\begin{remark}\label{rmk.difference}
    We would like to point out that the map $\zeta$ is not a special case of the local bijection $\psi$ presented in Section \ref{sec.localbij}. For the map $\zeta$, we are allowed to insert horizontal steps into a Dyck path in any position, but for the map $\psi$, the inserted positions depend on the marking rule.  
\end{remark}

For a Motzkin path $\omega$, a non-empty sequence of consecutive horizontal steps is called a \emph{wide plateau} of $\omega$ if it is preceded by an up step and followed by a down step in $\omega$. 

\begin{lemma} \label{lem.insertion}
Let $n=k+2\ell$. Given an ordered pair $(\alpha,\tau) \in \mathcal{T}(n,k) \times \mathcal{C}_{\ell}$, let $\alpha=(a_1,a_2,\dots,a_k)$ and $\tau_i$ be the $i$th step of $\tau$. Suppose $\omega_0= \tau, \omega_1, \dots,\omega_k=\omega$ is the sequence of Motzkin paths constructed from $(\alpha,\tau)$, and  $F_i=(f_i(0),f_i(1),\dots,f_i(2\ell))$ is the vector associated to $\omega_i$ under Algorithm A. We assume the image $\omega$ of $(\alpha,\tau)$ under the map $\zeta$ is factored as in (\ref{eqn:expression}), where
\begin{equation*} 
    \omega = \eta_0 \oplus \tau_1\oplus \eta_1 \oplus \tau_2 \oplus \eta_2 \oplus \cdots \oplus \tau_{2\ell} \oplus \eta_{2\ell}.
\end{equation*}
Then the following properties hold.
\begin{enumerate}
    \item[(1)] For $1\le i\le k$, we have
    \begin{equation*}
        a_i\in\{f_{i-1}(0),f_{i-1}(1),\dots,f_{i-1}(2\ell)\}=\{0,1,\dots,2\ell-1+i\}\setminus\{a_1,\dots,a_{i-1}\}.
    \end{equation*}
    \item[(2)] For $0\le i\le k-1$, if $\omega_{i+1}$ is obtained from $\omega_i$ by inserting a horizontal step in the section of horizontal steps $\eta_j$ of $\omega_i$, then $\mathsf{comaj}(\omega_{i+1}) = \mathsf{comaj}(\omega_i)+f_i(j)$.
    \item[(3)] We have $\mathsf{comaj}(\omega)=\mathsf{comaj}(\tau)+\mathsf{add}(\alpha)$.
    \item[(4)] For the Motzkin path $\omega$, suppose $\eta_j$ is non-empty for some $j\in\{0,1,\dots,2\ell\}$. Assume that $\eta_j$ goes from the $x$-coordinate $h$ to the $x$-coordinate $g$ for some $g>h$. Then the set $X(\eta_j)$ of integers associated to the steps in $\eta_j$ is
    \begin{equation} \label{eqn:non-empty-section}
        X(\eta_j)=
        \begin{cases}
        \{b_j\}\cup\{h+1+b_j, \dots, g-1+b_j\}, &\mbox{if $\eta_j$ is a wide plateau,}  \\
        \{h+1+b_j, \dots, g+b_j\}, &\mbox{if $\eta_j$ is followed by an up step,} \\
        \{h+b_j, \dots, g-1+b_j\}, &\mbox{otherwise.}
        \end{cases}
    \end{equation}
\end{enumerate}
\end{lemma}
\begin{proof} 
(1) Since $0\le a_1<a_2<\cdots a_k<n$, we have 
\begin{equation} \label{eqn:upper_bound}
    a_i\le n-1-k+i=2\ell-1+i,
\end{equation} 
for $1\le i\le k$. 
We shall prove the assertion (1) by induction. For the initial case $i=1$, by Algorithm A (1), we have $\{f_0(0), f_0(1), \dots, f_0(2\ell)\} = \{0,1,\dots,2\ell\}$. By (\ref{eqn:upper_bound}), $a_1\leq 2\ell$. Thus, $a_1\in\{f_0(0),f_0(1),\dots,f_0(2\ell)\}$. For $i\geq 1$, suppose the assertion (1) holds for the case $i$. By induction hypothesis, $a_i \in \{f_{i-1}(0), f_{i-1}(1), \dots, f_{i-1}(2\ell)\} = \{0,1,\dots,2\ell-1+i\} \setminus \{a_1,\dots,a_{i-1}\}$. By Algorithm A (2), it follows from (\ref{eqn:case-1}) and (\ref{eqn:case-2}) that
\begin{equation*}
    \{f_{i}(0),f_{i}(1),\dots,f_{i}(2\ell)\}=\big(\{f_{i-1}(0),f_{i-1}(1),\dots,f_{i-1}(2\ell)\}\setminus \{a_i\}\big)\cup\{2\ell+i\}.
\end{equation*}
By \eqref{eqn:upper_bound}, $a_{i+1}\leq 2\ell +i$, and hence $a_{i+1}\in\{f_{i}(0),f_{i}(1),\dots,f_{i}(2\ell)\}$. Thus, the assertion (1) follows.

(2) We shall prove the assertion (2) by induction. Consider the initial case $\omega_0=\tau$. For the peaks  $0<c_1<c_2< \cdots <c_d< 2\ell$ of $\tau$, if a horizontal step is inserted in the section $\eta_{c_t}$ for some $t$, then no weak valley is created, but the valleys on the right of $\eta_{c_t}$ are shifted to the right by one step. This operation results in an increment of $f_0(c_t)=b_{c_t}$ on $\mathsf{comaj}(\omega_0)$.
Suppose a horizontal step is inserted in the section $\eta_j$ for some $j\in\{0,1,\dots,2\ell\} \setminus \{c_1,c_2,\dots,c_d\}$. If $\eta_j$ is followed by an up step, then a weak valley $HU$ at $j+1$ is created; otherwise, $\eta_t$ is followed by a down step and a weak valley $DH$ at $t$ is created. In either case, by \eqref{eqn:each_entry}, this operation results in an increment of $f_0(j)$ on $\mathsf{comaj}(\omega_0)$.
Thus, the assertion (2) holds for the initial case $\omega_0$.

For $i\geq 1$, suppose the assertion (2) holds for $\omega_{i-1}$. By (1), $a_i \in \{f_{i-1}(0),f_{i-1}(1), \dots,f_{i-1}(2\ell)\}$. By Algorithm A (2), the path $\omega_i$ is constructed according to the following two cases of $a_i$.

Case 1. $a_i<d$. Then the integer $a_i$ is assigned to a peak, say $c_t$ for some $t$ (i.e., $f_{i-1}(c_t)=a_i$), and $\omega_i$ is obtained from $\omega_{i-1}$ by inserting a horizontal step in the section $\eta_{c_t}$. By induction hypothesis, $\mathsf{comaj}(\omega_i)=\mathsf{comaj}(\omega_{i-1})+f_{i-1}(c_t)=\mathsf{comaj}(\omega_{i-1})+a_i$. By \eqref{eqn:case-1}, the vector $F_i$ is obtained by setting
\begin{equation} \label{eqn:f0-1}
    f_i(h) = \begin{cases}
    c_t+1+b_{c_t}, & \mbox{if $h=c_t$,} \\
    f_{i-1}(h)+1,  & \mbox{if $f_{i-1}(h)\geq c_t+1+b_{c_t}$,} \\
    f_{i-1}(h),     & \mbox{otherwise,}
    \end{cases}
\end{equation}
for $0\le h\le 2\ell$.

We observe that if another horizontal step is inserted in $\eta_{c_t}$, then a weak valley $HH$ at $c_t+1$ will be created. Since $a_1<\cdots<a_i<d$, so far there is no weak valley on the right of $\eta_{c_t}$. This operation will result in an increment of $c_t+1+b_{c_t}=f_i(c_t)$ on $\mathsf{comaj}(\omega_i)$. For each section $\eta_h$ with $f_{i-1}(h)\geq c_t+1+b_{c_t}$, since $\eta_h$ is shifted to the right by one step, if a horizontal step is inserted in $\eta_h$, then it results in an increment of $f_{i-1}(h)+1=f_i(h)$ on $\mathsf{comaj}(\omega_i)$.

Case 2. $a_i\geq d$. Let $f_{i-1}(j)=a_i$ for some $j$.
Then $\omega_i$ is obtained from $\omega_{i-1}$ by inserting a horizontal step in the section $\eta_{j}$. By induction hypothesis, $\mathsf{comaj}(\omega_i)=\mathsf{comaj}(\omega_{i-1})+f_{i-1}(j)=\mathsf{comaj}(\omega_{i-1})+a_i$. By \eqref{eqn:case-2}, the vector $F_i$ is obtained by setting
\begin{equation} \label{eqn:f0-2}
    f_i(h) = \begin{cases}
    f_{i-1}(h)+1,  & \mbox{if $f_{i-1}(h)\geq a_i$,} \\
    f_{i-1}(h),     & \mbox{otherwise,}
    \end{cases}
\end{equation}
for $0\le h\le 2\ell$. 

Note that if there is an index  $s\in\{1,\dots,i-1\}$ such that $d\leq a_s<a_i$, then the section, say $\eta_g$, with $f_{i-1}(g)=a_s$ is on the left of $\eta_j$. Hence there is no weak valley on the right of $\eta_j$. Thus, for each section $\eta_h$ with $f_{i-1}(h)>a_i$, since $\eta_h$ is shifted to the right by one step, if a horizontal step is inserted in $\eta_h$, then it results in an increment of $f_{i-1}(h)+1=f_i(h)$ on $\mathsf{comaj}(\omega_i)$.

Thus, the assertion (2) follows by induction.

(3) By assertions (1) and (2), we have $\mathsf{comaj}(\omega_0)=\mathsf{comaj}(\tau)$ and $\mathsf{comaj}(\omega_i)=\mathsf{comaj}(\omega_{i-1})+a_i$ for $1\leq i\leq k$. Therefore, the result follows.

(4) Note that if the section $\eta_j$ is a wide plateau of $\omega$, then $j$ is a peak of the Dyck path $\tau$. By \eqref{eqn:expression}, $b_j\in X(\eta_j)$. Moreover, $\eta_j$ contains weak valleys $HH$ at $h+1, h+2, \dots, g-1$. Hence $X(\eta_j)=\{b_j\}\cup\{h+1+b_j,\dots,g-1+b_j\}$. 
If the section $\eta_j$ is followed by an up step, then $\eta_j$ contains a weak valley $HU$ at $g$ and weak valleys $HH$ at $h+1, h+2,\dots, g-1$. Hence $X(\eta_j)=\{h+1+b_j,\dots, g+b_j\}$. (In the case that $\eta_j$ is preceded by a down step and followed by an up step, the weak valley $DH$ at $h$ has been counted as a valley of the Dyck path $\tau$.) If $\eta_j$ is not a plateau or followed by an up step, then it must be preceded by a down step. It follows that $\eta_j$ contains a weak valley $DH$ at $h$ and weak valleys $HH$ at $h+1,\dots, g-1$. Hence $X(\eta_j)=\{h+b_j,\dots,g-1+b_j\}$. The result follows.
\end{proof}

Let $n=k+2\ell$. Now, we describe the map $\zeta^{\prime}: \mathcal{M}(n,k) \rightarrow \mathcal{T}(n,k) \times \mathcal{C}_{\ell}$. Given a Motzkin path $\omega\in\mathcal{M}(n,k)$, we will construct an ordered pair $(\alpha,\tau) \in \mathcal{T}(n,k) \times \mathcal{C}_{\ell}$ from $\omega$. If $k=n$, then set $\alpha=(0,1,\dots,n-1)$ and $\tau$ is empty. If $k=0$, then $\alpha$ is empty and set $\tau=\omega$. For $1\le k\le n-1$, the ordered pair $(\alpha,\tau)$ is constructed by the following algorithm.

\medskip
\noindent {\bf Algorithm B.}

We decompose $\omega$ as
    \begin{equation} \label{eqn:factorization}
        \omega = \eta_0 \oplus x_1 \oplus \eta_1 \oplus x_2 \oplus \eta_2 \oplus \cdots \oplus x_{2\ell} \oplus \eta_{2\ell},
    \end{equation}
    where each $\eta_i$ is a sequence of consecutive horizontal steps (possibly empty), and each $x_i$ is either an up step or a down step. The Dyck path $\tau$ is obtained from $\omega$ by directly removing the $k$ horizontal steps. We may write 
    \begin{equation} \label{eqn:hidden-Dyck-path}
        \tau= x_{1} \oplus x_{2} \oplus \cdots \oplus x_{2\ell}.
    \end{equation} 
Let $(b_0,b_1,\dots, b_{2\ell})$ be the tuple of integers associated with $\tau$ defined in \eqref{eq.vector}. It remains to determine the sequence $\alpha$, which is obtained from an integer-labeling of the $k$ horizontal steps of $\omega$ as follows. 

\begin{enumerate}
\item[(1)] For each wide plateau of $\omega$, say $\eta_h$ for some $h\in\{1,\dots,2\ell-1\}$, we mark the last step of $\eta_h$ and label this step with the integer $b_h$. 
\item[(2)] For each unmarked horizontal step $z$ of $\omega$, let $z$ be the $i$th step of $\omega$ for some $i$, and let $\eta_j$ be the section of horizontal steps containing $z$ for some $j$.  If $\eta_j$ is a wide plateau or is followed by an up step, then we label $z$ with $i+b_j$; otherwise, we label $z$ with $i-1+b_j$. 
\end{enumerate}
The announced sequence $\alpha$ is the increasing arrangement of the labels on these horizontal steps.

\begin{example} 
    We consider the Motzkin path $\omega \in \mathcal{M}(15,5)$ shown in Figure \ref{fig.exalgoB1}. Figure \ref{fig.exalgoB2} depicts the resulting Dyck path $\tau \in \mathcal{C}(5)$ obtained from $\omega$ by removing all its horizontal steps. The vector that $\tau$ is associated with is given by $(b_0,b_1,\dots,b_{10})=(2,2,2,1,1,1,1,1,0,0,0)$. As shown in \eqref{eqn:factorization}, we have $w= \eta_0 \oplus x_1 \oplus \eta_1  \oplus \cdots \oplus x_{10} \oplus \eta_{10}$, where $\eta_2, \eta_5, \eta_8$, and $\eta_{10}$ are non-empty. To determine the sequence $\alpha$, we observe that $\eta_2$ and $\eta_5$ belong to Algorithm B $(1)$; their last horizontal steps are labeled (in red in Figure \ref{fig.exalgoB1}) as $b_2$ and $b_5$, respectively. Next, the horizontal steps located at steps $3$ ($i=3$ and $j=2$) and $12$ ($i=12$ and $j=8$) of $w$ have the form mentioned in Algorithm B $(2)$, they are labeled by $3+b_2 = 5$ and $12+b_8=12$, respectively. Finally, we label the last horizontal step by $15-1+b_{10} = 14$. We then obtain the sequence $\alpha=(1,2,5,12,14)$.
\end{example}
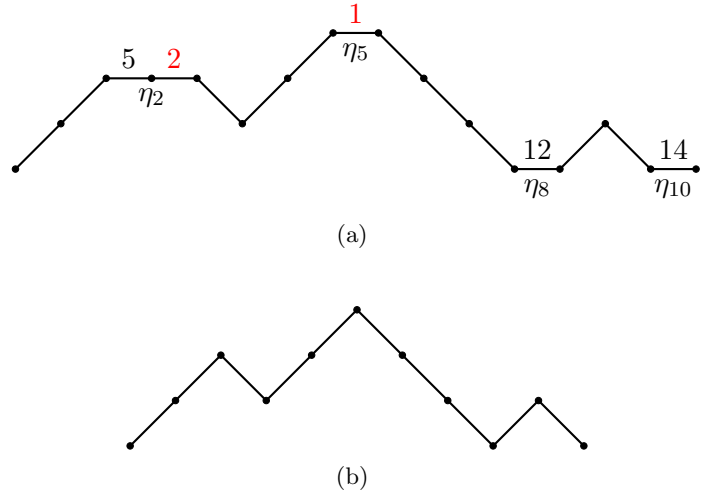
\begin{figure}[hbt!]
        \centering
        \subfigure[]{\label{fig.exalgoB1}
            \begin{tikzpicture}[scale=0.6]
                \draw[thick, black]
                    (0,0) -- (1,1)
                          -- (2,2)
                          -- (3,2)
                          -- (4,2)
                          -- (5,1)
                          -- (6,2)
                          -- (7,3)
                          -- (8,3)
                          -- (9,2)
                          -- (10,1)
                          -- (11,0)
                          -- (12,0)
                          -- (13,1)
                          -- (14,0)
                          -- (15,0);

                \node[above] at (2.5, 2) {$5$};
                \node[above] at (3.5, 2) {\textcolor{red}{$2$}};
                \node[above] at (7.5, 3) {\textcolor{red}{$1$}};
                \node[above] at (11.5, 0) {$12$};
                \node[above] at (14.5, 0) {$14$};

                \node[below] at (3, 2) {$\eta_2$};
                \node[below] at (7.5, 3) {$\eta_5$};
                \node[below] at (11.5, 0) {$\eta_8$};
                \node[below] at (14.5, 0) {$\eta_{10}$};

                \foreach \x/\y in {0/0, 1/1, 2/2, 3/2, 4/2, 5/1, 6/2, 7/3, 8/3, 9/2, 10/1, 11/0, 12/0, 13/1, 14/0, 15/0}{
                    \filldraw[black] (\x,\y) circle (2pt);
                }    
            \end{tikzpicture}
        }
        
        \vspace{1em}
        \subfigure[]{\label{fig.exalgoB2}
            \begin{tikzpicture}[scale=0.6]
                \draw[thick, black]
                    (0,0) -- (1,1)
                          -- (2,2) 
                          -- (3,1)
                          -- (4,2)
                          -- (5,3)
                          -- (6,2)
                          -- (7,1)
                          -- (8,0)
                          -- (9,1)
                          -- (10,0);
    
                \foreach \x/\y in {0/0, 1/1, 2/2, 3/1, 4/2, 5/3, 6/2, 7/1, 8/0, 9/1, 10/0}{
                    \filldraw[black] (\x,\y) circle (2pt);
                }    
            \end{tikzpicture}
        }
        \caption{An illustration of the map $\zeta^{\prime}$ via Algorithm B.}
        \label{fig.exalgoB}
    \end{figure}

\begin{lemma} \label{lem.inversezeta}
Let $n=k+2\ell$. For any Motzkin path $\omega \in \mathcal{M}(n,k)$, let $(\alpha,\tau)$ be the ordered pair obtained from $\omega$ under $\zeta^{\prime}$ using Algorithm B. Then $(\alpha,\tau) \in \mathcal{T}(n,k) \times \mathcal{C}(\ell)$.
\end{lemma}
\begin{proof}
The cases for $k=0$ and $k=n$ are clear. 
For $1\le k\le n-1$, let $\omega$ be decomposed as in \eqref{eqn:factorization}. By Algorithm B, the Dyck path $\tau$ is obtained from $\omega$ by removing the $k$ horizontal steps. Hence $\tau \in \mathcal{C}(\ell)$. Moreover, the $k$ horizontal steps are labeled by distinct non-negative integers. The sequence $\alpha=(a_1,a_2,\dots,a_k)$ is obtained by arranging these labels in increasing order. Suppose the greatest label $a_k$ appears in the section $\eta_j$ for some $j$. By Algorithm B (1), if $\eta_j$ is a wide plateau, then $j\leq 2\ell-1$. We observe that $a_k\leq n-2$ (since the last step of $\eta_j$ is marked). By Algorithm B (2), if $\eta_j$ is followed by an up step, then $j\leq 2\ell-2$. We observe that $a_k\leq n-2$. If $\eta_j$ is neither a wide plateau nor followed by an up step, then $j\leq 2\ell$ and $a_k\leq n-1$. Hence $\alpha\in\mathcal{T}(n,k)$. The result follows.
\end{proof}

\begin{proof}[Proof of Theorem \ref{thm.tuple-Dyck-to-Motzkin}]

    It suffices to consider the case $1\le k\le n-1$. Given an ordered pair $(\alpha,\tau)\in\mathcal{T}(n,k) \times \mathcal{C}(\ell)$, where $\alpha=(a_1,a_2,\dots,a_k)$, let $\omega$ be the Motzkin path constructed from $(\alpha,\tau)$. 
    By Lemma \ref{lem.insertion} (1)-(2), for $1\leq i\leq k$, each $a_i$ is associated to a unique horizontal step in some section $\eta_j$ of horizontal steps of $\omega$, where $j\in\{0,1,\dots,2\ell\}$. Thus, $\omega$ is a Motzkin path of size $n=k+2\ell$ containing $k$ horizontal steps, that is, $\zeta: (\alpha,\tau) \mapsto \omega \in \mathcal{M}(n,k)$. By Lemma \ref{lem.insertion} (3), we have $\mathsf{comaj}(\omega)=\mathsf{comaj}(\tau)+\mathsf{add}(\alpha)$. Moreover, each non-empty section $\eta_j$ has the property presented in Lemma \ref{lem.insertion} (4). This property shows that, under Algorithm B, the map $\zeta^{\prime}: \omega \mapsto (\alpha, \tau)$. Thus, $\zeta^{\prime} \circ \zeta$ is an identity function on $\mathcal{T}(n,k) \times \mathcal{C}(\ell)$.

    On the other hand, given a Motzkin path $\omega \in \mathcal{M}(n,k)$, let $(\alpha,\tau)$ be the ordered pair constructed from $\omega$ via Algorithm B. By Lemma \ref{lem.inversezeta}, we have $\zeta^{\prime}:\omega \mapsto (\alpha,\tau) \in \mathcal{T}(n,k) \times \mathcal{C}(\ell)$. Let $\omega^{\prime}$ be the image of $(\alpha,\tau)$ under the map $\zeta$ constructed via Algorithm A. Then each non-empty section $\eta_j$ of horizontal steps of $\omega^{\prime}$ has the property in Lemma \ref{lem.insertion} (4). It follows that $\omega^{\prime}=\omega$ with horizontal steps having the same labels. Thus, $\zeta \circ \zeta^{\prime}$ is an identity function on $\mathcal{M}(n,k)$. Therefore, the map $\zeta$ is a bijection, with $\zeta^{\prime}$ the inverse of $\zeta$. The announced bijection in Theorem \ref{thm.tuple-Dyck-to-Motzkin} is established.

\end{proof}

\begin{proof}[The second proof of Theorem \ref{thm.Guoconj}.]
    Let $0 \leq k \leq n$ with $n \equiv k \pmod{2}$. We set $n = k + 2 \ell$. It is well-known\footnote{This result was first due to MacMahon \cite{MacMahon}; see also \cite[Theorem 1.6]{Haglund08}. We note that the major statistic in the setting of \cite[Theorem 1.6]{Haglund08} translates to our comajor statistic on Dyck paths.} that the generating function of the set of Dyck paths $\mathcal{C}(\ell)$ by the comajor statistic is given by the $q$-Catalan number
    \begin{equation}
        \sum_{\tau \in \mathcal{C}(\ell)} q^{\mathsf{comaj}(\tau)} = \frac{1}{[\ell+1]_q} \begin{bmatrix}
            2\ell \\ \ell
        \end{bmatrix}_q
        :=C_\ell(q).
    \end{equation}

    It follows from a simple change of variable in \eqref{eq.qadd} that the $q$-enumeration formula for the set $\mathcal{T}(n,k)$ by $\mathsf{add}$ is given by
    \begin{equation}\label{eq.qaddstrict}
        \sum_{\alpha \in \mathcal{T}(n,k)} q^{\mathsf{add}(\alpha)} = q^{\binom{k}{2}}
        \begin{bmatrix}
            n \\ k
        \end{bmatrix}_q.
    \end{equation}

    By Theorem \ref{thm.tuple-Dyck-to-Motzkin}, 
    \begin{align*}
        \sum_{w \in \mathcal{M}(n,k)} q^{\mathsf{comaj}(w)} &= \sum_{(\alpha, \tau) \in \mathcal{T}(n,k) \times \mathcal{C}(\ell)} q^{\mathsf{add}(\alpha) + \mathsf{comaj}(\tau)} \\
        &= \left( \sum_{\alpha \in \mathcal{T}(n,k) } q^{\mathsf{add}(\alpha)} \right) \left( \sum_{\tau \in \mathcal{C}(\ell)} q^{\mathsf{comaj}(\tau)} \right) \\
        &= q^{\binom{k}{2}} \begin{bmatrix}
            n \\ k
        \end{bmatrix}_q
        C_{\ell}(q),
    \end{align*}
    as desired. This completes the second proof of Theorem \ref{thm.Guoconj}.
\end{proof}

\section{Concluding remarks}\label{sec.remarks}

In summary, our shape algorithm determines the shape of the Richardson tableau from the corresponding Motzkin path without applying the Robinson--Schensted algorithm. Keeping track of the major and comajor statistics under the local bijection by inserting unit Motzkin paths gives a combinatorial proof of the $q$-enumeration formula for Richardson tableaux of a fixed shape, together with its comajor counterpart. Also, we prove Guo's conjecture on Richardson tableaux with a prescribed number of odd columns in two ways.

We conclude with a few questions. One may wonder if there is a refinement that keeps track simultaneously of the shape and the number of odd columns, together with the major or comajor statistic. Such a refinement would contain both of our main $q$-enumeration results as special cases. Also, since Richardson tableaux arise naturally from the geometry of Springer fibers, it would be interesting to understand whether the shape algorithm or the local bijection has a direct geometric interpretation.
\vspace{1em}


\noindent \textbf{Acknowledgements.} This research was supported by the National Science and Technology Council, Taiwan, through grants 113-2115-M-003-010-MY3 (S.-P. Eu), 113-2115-M-153-002 (T.-S. Fu), and 114-2811-M-003-024 (Y.-L. Lee).



\bibliographystyle{plain}
\bibliography{Richardson}



\end{document}